\documentclass[11pt]{article}
\usepackage{amsmath, latexsym, amsfonts, amssymb, amsthm, amscd}
\usepackage{mathtools} 
\usepackage[left=2cm, right=2cm, top=2.5cm, bottom=2.5cm]{geometry}
\usepackage{upquote} 
\usepackage{color}
\usepackage{setspace} 
\usepackage[labelfont=bf]{caption} 
\usepackage{stmaryrd} 
\usepackage[title]{appendix}

\usepackage{kpfonts} 
\usepackage{dsfont} 

\usepackage[english]{babel} 

\usepackage[utf8]{inputenc}
\usepackage[T1]{fontenc}

\usepackage{bm} 

\usepackage{enumitem} 

\usepackage{xcolor}
\definecolor{Maroon}{HTML}{ad2231}
\definecolor{webgreen}{HTML}{008000}
\usepackage{hyperref}
\hypersetup{colorlinks, breaklinks, urlcolor=Maroon, linkcolor=Maroon, citecolor=webgreen} 

\allowdisplaybreaks

\usepackage{tikz}
\usetikzlibrary{calc}
\usepackage{pgfplots}

\newtheorem{theorem}{Theorem}[section]
\newtheorem{corollary}[theorem]{Corollary}
\newtheorem{proposition}[theorem]{Proposition}
\newtheorem{lemma}[theorem]{Lemma}
\newtheorem{remark}[theorem]{Remark}

\theoremstyle{definition}

\numberwithin{equation}{section}

\usepackage{todonotes}

\begin{document}
\title{On the speed of coming down from infinity for subcritical branching processes with pairwise interactions}
\author{Gabriel Berzunza Ojeda\footnote{ {\sc Department of Mathematical Sciences, University of Liverpool. Liverpool, United Kingdom.} E-mail: gabriel.berzunza-ojeda@liverpool.ac.uk}\,\, 
and
Juan Carlos Pardo\footnote{{\sc Centro de Investigaci\'on en Matem\'aticas A.C., Mexico.} E-mail: jcpardo@cimat.mx} \\ \vspace*{10mm}
}
\date{ }
\maketitle

\vspace{0.1in}

\begin{abstract} 
In this paper, we study the phenomenon of {\it coming down from infinity} for subcritical cooperative branching processes with pairwise interactions (BPI processes) under suitable conditions. BPI processes are continuous-time Markov chains that  extend  classical branching models by incorporating additional mechanisms accounting for both competitive and cooperative interactions  between pairs of individuals. 

Our main focus is  on characterising the speed at which BPI processes evolve when starting from a very large initial population in the subcritical regime. In addition, we investigate their second-order fluctuations. Furthermore, our results also apply to a class of exchangeable fragmentation-coalescent processes introduced by Berestycki  \cite{BerestyckiJ2004} and several other models from population genetics.
\end{abstract}


\noindent {\sc Subject Classes}: 60J80, 60K35, 60F17, 60J60, 90D25


\section{Introduction} \label{Sec1}

In this paper, we explore the speed of \textit{coming down from infinity} for subcritical cooperative branching processes with pairwise interactions (referred to as BPI processes), under appropriate conditions. These processes were recently introduced in their current form by Gonz\'alez-Casanova et al.\ \cite{Pa2021} and further analysed by the authors in \cite{Berzunza2020}. Our primary focus is on understanding their short-time behaviour when the initial population size is significantly large

BPI processes emerge as a natural extension of (continuous-time) Bienaym\'e-Galton-Watson processes (BGW processes), addressing some of their inherent limitations, in particular their unrealistic properties. By incorporating pairwise interactions such as competition and cooperation between individuals, BPI processes can capture complex dynamics observed in biological systems. This generalisation has been studied in various contexts by several authors, including Jagers \cite{Jagers1994}, Kalinkin \cite{Ka82, Ka99, Ka01, Ka02, Ka02-1}, Lambert \cite{Lambert2005}, and more recently by Gonz\'alez-Casanova et al.\ \cite{Pa2021} and the authors \cite{Berzunza2020}. In these works, the birth and death rates are generalised by considering polynomial rates as functions of the population size which can be interpreted as different types of interactions between individuals.

Formally, let  $\mathbb{N}\coloneqq\{1,2,\ldots\}$ and $\mathbb{N}_{0} \coloneqq \mathbb{N} \cup \{0\}$. BPI processes are continuous-time Markov chains with state space $\mathbb{N}_{0}$. They feature natural births and deaths at rates proportional to the current population size, similar to the continuous-time BGW process, but also incorporate additional birth and death events due to cooperation and competition, occurring at rates proportional to the square of the population size. Specifically, a BPI process with parameters $c, d \in \mathbb{R}_{+}:=[0,\infty)$ and sequences $(b_{i})_{i \geq 1}$ and $(\pi_{i})_{i \geq 1}$ where $b_i, \pi_i\in \mathbb{R}_{+}$, for $i\ge 1$,  and  $b \coloneqq \sum_{i \geq 1} b_{i}  < \infty$ and  $\rho \coloneqq \sum_{i \geq 1} \pi_{i}  < \infty$, is defined by the infinitesimal generator  $Q = (q_{i,j})_{i,j \in \mathbb{N}_{0}}$ given by
\begin{align} \label{Qmatrix}
q_{i,j} = \left\{ \begin{array}{ll}
 i \pi_{j-i} + i(i-1)b_{j-i} & \mbox{  if }  i \geq 1 \, \, \text{and} \, \, j >i, \\
 di +c i (i-1)  & \mbox{  if }  i \geq 1 \, \, \text{and} \, \, j = i-1, \\
 -i(d + \rho + (b + c) (i-1)  )  & \mbox{  if }  i \geq 1 \, \, \text{and} \, \, j =i,  \\
 0 &  \text{otherwise}. \\
    \end{array}
    \right.
\end{align}
\noindent For $z \in \mathbb{N}_{0}$, we denote by $(Z_{t}^{(z)})_{t \geq 0}$ a BPI process issued from $z$ (i.e., $Z_{0}^{(z)} = z$). 

From the infinitesimal generator   \eqref{Qmatrix}, or $Q$-matrix, we observe that the additional death rate $cz(z-1)$ represents competition, where one of the  $z$
 individuals selects another from the remaining  $z-1$
 at a constant rate $c$ and kills it. The additional birth rate  $bz(z-1)$
 reflects cooperative interactions, where pairs of individuals produce 
$i$ new individuals with probability $b_i/b$.\\

Let 
\begin{align}
\mathbf{m}_{\rm int} \coloneqq-c+\sum_{i \geq 1} ib_{i}.
\end{align}
According to Gonz\'alez-Casanova et al.\ \cite{Pa2021}, this quantity influences the long-term behaviour of BPI processes. A BPI process is termed supercritical, critical, or subcritical cooperative depending on whether  $\mathbf{m}_{\rm int}>0$, $\mathbf{m}_{\rm int}=0$ or $\mathbf{m}_{\rm int}<0$, respectively.  Theorem 1 in \cite{Pa2021} shows that in the supercritical regime, and even under
\begin{align}
 \mathbf{m}_{\rm br} \coloneqq-d+\sum_{i \geq 1} i\pi_{i} < \infty,
\end{align}
\noindent the BPI process may explode in finite time with positive probability. In the subcritical and critical cooperative regime, and under  $\mathbf{m}_{\rm br}<\infty$, the  BPI process does not explode in finite time almost surely. More recently, the authors  \cite{Berzunza2020} studied the explosion event for general subcritical and critical cooperative BPI processes (i.e.\ not necessarily satisfying that $\mathbf{m}_{\rm br}<\infty$) and provide general integral tests, see Theorems 1.1 and 1.4 in \cite{Berzunza2020} for further details.

When $c=b=0$, the BPI process reduces to the  BGW process. A special case of BPI processes is the so-called logistic branching process (LB-process), deeply studied in \cite{Lambert2005}, which corresponds to $c > 0$ and $b = 0$. Another example is the block counting process of the renowned Kingman coalescent, which corresponds to a BPI process with no branching or cooperation events (i.e., $c>0$ and $d = \rho = b = 0$); see \cite{Kingman1982}. More generally, when $c, \rho >0$ and $d=b=0$, the BPI-process is exactly the block counting process of an exchangeable fragmentation-coalescence (EFC for short) process, introduced in \cite{BerestyckiJ2004}, where the coagulation mechanism features only the Kingman component. Moreover, the BPI process can be viewed as the counting process of a particle system where particles can coagulate (via competition interactions), fragment without collision (via branching), or fragment due to collisions (via cooperative interactions). Ultimately, when $\pi_i = b_i = 0$ for all $i \ge 2$, the BPI process reduces to a classic birth-and-death process.

Motivated by studies on population dynamics with initially large populations, the phenomenon of {\it coming down from infinity} has attracted significant attention. This phenomenon describes a scenario where a random process starting from an infinite number of individuals (or blocks) immediately reaches a finite state. Kingman \cite{Kingman1982} was the first to observe this phenomenon in the context of the Kingman coalescent. Subsequently, Schweinsberg \cite{Schweinsberg2000} provided necessary and sufficient conditions for the occurrence of this phenomenon in the setting of the $\Lambda$-coalescent. Further work, for the $\Lambda$-coalescent, includes studies of the speed of coming down from infinity by Berestycki et al.\ \cite{Bere2010} and second-order asymptotics by Limic and Talarczyk \cite{Vlada20152, Vlada2015}. The case of exchangeable fragmentation-coalescence processes (EFC processes) was treated in Berestycki \cite{BerestyckiJ2004} where  a sufficient condition for the process to come down from infinity is provided.
 Recent studies include those by Foucart \cite{Foucart2022}, exploring the phenomenon of coming down from infinity for EFC processes with multiple coagulations but not simultaneous, as in the $\Lambda$-coalescent, and the fragmentation dislocates, at a finite rate, an individual block into sub-blocks of infinite size. Bansaye et al.\ \cite{Bansaye2016} described the speed of coming-down from infinity for birth-and-death processes that eventually become extinct. More recently, Bansaye \cite{Bansaye2019} developed  a general framework for the phenomenon across various stochastic processes with jumps described by  stochastic differential equations. Finally, it is important to note that the phenomenon of coming down from infinity for subcritical cooperative BPI processes was first studied by  Gonz\'alez-Casanova et al.\ \cite{Pa2021} under the assumption that $ \mathbf{m}_{\rm br} < \infty$. The authors \cite{Berzunza2020} established a sufficient condition for general subcritical and critical cooperative BPI processes to come down from infinity. However, nothing is known about the speed  at which BPI processes come down from infinity.
 
Our aim is to determine the speed at which a subcritical BPI process comes down from infinity. In the subcritical cooperative regime (i.e.\ $\mathbf{m}_{\rm int}<0$), under the assumption that $\mathbf{m}_{\rm br}<\infty$ and the cooperation parameters $(b_{i})_{i \geq 1}$ have a bit more than a finite first moment (see \eqref{Assump3TWO}), we show that the speed of coming down from infinity is asymptotically equivalent to $(-\mathbf{m}_{\rm int}t)^{-1}$. This is similar to the Kingman coalescent case where $-\mathbf{m}_{\rm int}=c$. 
 
Finally, we also show, under the assumption that the branching and cooperation parameters have finite large positive exponential moments, the convergence in the $\mathbb{L}_p$ sense (for $p\in(0,2]$), of the subcritical cooperative BPI process rescaled by its speed of coming down from infinity. Moreover, we also obtain the second-order fluctuations around the speed of coming down from infinity  at small times.

\section{Main results} \label{MainResultSec}

We first present our main results on the speed at which subcritical cooperative BPI processes come down from infinity. In the second part of this section, we examine the second-order fluctuations around this speed, assuming the existence of large exponential moments for the branching and cooperation parameters.

\subsection{The speed of coming down from infinity}

We henceforth consider the subcritical cooperative regime, i.e., when $\mathbf{m}_{\rm int}<0$. Additionally, we assume that the branching parameter has a finite first moment, i.e., $ \mathbf{m}_{\rm br}<\infty$. Under these assumptions, we have  $c>0$,  $\sum_{i \geq 1} ib_{i} < \infty$ and  $\sum_{i \geq 1} i\pi_{i} < \infty$. 
Furthermore, the BPI process does not explode in finite time; see for instance Example 2.1 in \cite{Berzunza2020} or Theorem 1 in \cite{Pa2021}.

We also require the following moment condition on the cooperative parameters,
\begin{equation} \label{Assump3TWO}
\text{there exists} \quad \alpha \in (1,2] \quad \text{such that} \quad \sum_{i=1}^{\infty} i^{\alpha} b_{i} < \infty. \tag{${\bf H1}$}
\end{equation}
Examples of sequences that satisfy the previous hypothesis include regularly varying sequences. More precisely, let  $\beta \in (1,2]$ such that $b_{i} = i^{-\beta -1} L(i)$, as $i \rightarrow \infty$, where $L: \mathbb{R}_{+} \rightarrow \mathbb{R}_{+}$ is a slowly varying function at infinity. According to Lemma 4.4 in \cite{BJS2015},  for any $\delta \in (1,\beta)$, 
\begin{align}
\sum_{i \geq n} i^{\delta}b_{i} \sim (\beta -\delta)^{-1} n^{-(\beta -\delta)} L(n), \qquad \textrm{as}\qquad n \rightarrow \infty.
\end{align}
By using, for example, (1) in Lemma 4.1 in \cite{BJS2015} with $\varepsilon \in (0, \beta-\delta)$, we find that $\sum_{i \geq n} i^{\delta}b_{i} \rightarrow 0$, as $n \rightarrow \infty$. This implies that $\sum_{i \geq 1} i^{\delta}b_{i} < \infty$, thus satisfying  \eqref{Assump3TWO} with $\alpha =\delta$.\\

We  equip $\overline{\mathbb{R}}:=\mathbb{R}\cup \{\infty\}$ with the distance $\bar{d}(x,y) := |e^{-x}-e^{-y}|$, for $x,y \in \overline{\mathbb{R}}$. Let $\mathbb{D}(\mathbb{R}_{+}, \overline{\mathbb{R}})$ be the space of c\`adl\`ag functions from  $\mathbb{R}_{+} \coloneqq [0, \infty)$ to $\overline{\mathbb{R}}$ and endowed  with the Skorokhod ${\rm J}_{1}$ topology (see e.g., \cite[Chapter 3]{Billingsley1999}). 
\begin{theorem} \label{Main1}
Suppose that  $\mathbf{m}_{\rm int}<0$, $\mathbf{m}_{\rm br}<\infty $ and \eqref{Assump3TWO} hold. Then, the family $(Z^{(z)})_{z \in \mathbb{N}_{0}}$ of BPI processes converges weakly, in $\mathbb{D}(\mathbb{R}_{+}, \overline{\mathbb{R}})$, as $z \rightarrow \infty$, towards $Z^{(\infty)} = (Z^{(\infty)}_{t})_{t \geq 0}$, that is, a BPI process starting from $\infty$. Moreover,
\begin{align}
\mathbb{P}\Big(Z^{(\infty)}_{t} < \infty, \, \,\,\, \text{ for all } \, \, t >0\Big) = 1
\end{align}
and
\begin{align} \label{MainConv}
 \lim_{t \downarrow 0}  (-\mathbf{m}_{\rm int})t\,Z_{t}^{(\infty)} = 1, \qquad \text{almost surely}.
\end{align}
\end{theorem}

When branching and cooperation are absent (i.e., $d = \rho = 0$ and $b = 0$), Theorem \ref{Main1} recovers the well-known result stating that the number of blocks in Kingman's coalescent is asymptotically equivalent to $(c t)^{-1}$, almost surely as $t \downarrow 0$ (see, e.g., \cite{Bere2010}). Recall that the BPI process corresponds to the block counting process of an EFC process with a coagulation mechanism featuring only the Kingman component and a fragmentation measure with no mass on partitions with infinitely many blocks, i.e.\ $c, \rho > 0$ and $d = b = 0$. Then, Theorem \ref{Main1} provides the first result on   the speed of coming down from infinity for EFC processes extending the known result in Proposition 15 of \cite{BerestyckiJ2004} (see also Corollary 1.2 in \cite{Foucart2022}), which asserts that the EFC process comes down from infinity. In general, under the assumptions   $\mathbf{m}_{\rm int}<0$, $\mathbf{m}_{\rm br}<\infty $ and \eqref{Assump3TWO}, Theorem \ref{Main1} establishes  that the Kingman component predominantly determines the speed at which the  subcritical cooperative BPI process comes down from infinity. \\

Our second main result in this section shows that the convergence \eqref{MainConv} in Theorem \ref{Main1} also holds in the $\mathbb{L}_{p}$ sense, for $p \in (0,2]$. However, this result  requires the following additional assumption, 
\begin{equation} \label{Assump4}
\textrm{there exists}\quad \vartheta >e^{288}\quad  \textrm{such that } \quad \sum_{i \geq 1} \vartheta^{i} \pi_{i} < \infty\quad \textrm{ and }\quad \sum_{i \geq 1} \vartheta^{i} b_{i} < \infty. \tag{${\bf H4}$}
\end{equation}
\noindent It is important to emphasize that the previous assumption can be made for any value of $\vartheta>1$, though this introduces a dependency of $\vartheta$ on the values of $p$ for convergence within $\mathbb{L}_{p}$.

Observe that \eqref{Assump4} implies that $\boldsymbol{\sigma}^{2}_{{\rm br}} \coloneqq d + \sum_{i \geq 1} i^{2} \pi_{i} <\infty$ (and thus $\mathbf{m}_{\rm br} <\infty$). Similarly, \eqref{Assump4} also implies that $\boldsymbol{\sigma}^{2}_{{\rm int}} \coloneqq c+\sum_{i=1}^{\infty} i^{2} b_{i} < \infty$ (i.e., \eqref{Assump3TWO} holds with $\alpha =2$).
\begin{theorem} \label{PropC2}
Suppose that  $\mathbf{m}_{\rm int}<0$ and  \eqref{Assump4} hold. Then, for $p \in (0,2]$, we have that
\begin{align} 
\lim_{t \downarrow 0}\mathbb{E}\Big[ \sup_{s \leq t}  \Big| - \mathbf{m}_{\rm int} s\, Z_{s}^{(\infty)}- 1 \Big|^{p} \Big] = 0. 
\end{align}
\end{theorem}

Indeed, the result in Theorem \ref{PropC2} holds for any $p>0$, provided that \eqref{Assump4} is satisfied for sufficiently large $\vartheta$ (see Section \ref{ProofofPropC2} for further details).

\subsection{Second order fluctuations}

Our final main result investigates the second-order fluctuations of $- \mathbf{m}_{\rm int} t\, Z_{t}^{(\infty)}$ for small times. More precisely, we will study the asymptotic behaviour in a functional sense, as $\varepsilon \downarrow 0$, of the process $(X_{t}^{(\varepsilon)})_{t \geq 0}$, for $\varepsilon >0$, defined by
\begin{align} \label{FlucD}
X_{0}^{(\varepsilon)}=0 \quad \text{and} \quad  X_{t}^{(\varepsilon)} = \varepsilon^{-\frac{1}{2}} \Big( - \mathbf{m}_{\rm int} \varepsilon t\, Z_{\varepsilon t}^{(\infty)} - 1 \Big), \quad \text{for} \quad t >0.  
\end{align}
\noindent In the next result, let $\mathbb{D}(\mathbb{R}_{+}, \mathbb{R})$ denote the space of c\`adl\`ag functions from  $\mathbb{R}_{+}$ to $\mathbb{R}$ (here, $\mathbb{R}$ is equipped with the Euclidean distance), endowed with the Skorokhod ${\rm J}_{1}$ topology. 

\begin{theorem} \label{Main3}
Suppose that  $\mathbf{m}_{\rm int}<0$ and \eqref{Assump4} hold. The process $(X_{t}^{(\varepsilon)})_{t \geq 0}$ converges weakly, in $\mathbb{D}(\mathbb{R}_{+}, \mathbb{R})$, as $\varepsilon \downarrow 0$, to a Gaussian process $(X_{t})_{t \geq 0}$ defined by
\begin{align}
X_{0}=0 \quad \text{and} \quad  X_{t} = \frac{\boldsymbol{\sigma}_{{\rm int}}}{t} \int_{0}^{t} u \,{\rm d}W_{u}, \quad \text{for} \quad t >0, 
\end{align}
\noindent where $(W_{t})_{t \geq 0}$ is a standard Brownian motion. 
\end{theorem}

Theorem \ref{Main3} recovers the second-order asymptotic result of the number of blocks in Kingman's coalescent (i.e., when $d = \rho = 0$ and $b = 0$) about its speed of coming down from infinity (Theorem 1.1 in \cite{Vlada20152}). The process $(X_{t})_{t \geq 0}$ is similar to the process in Theorem 1.3 of \cite{Vlada20152} in the case of the number of blocks of a $\Lambda$-coalescent with a Kingman part.

\subsection{Examples from population genetics} \label{Examples}

Kingman's coalescent \cite{Kingman1982} is widely regarded as the simplest stochastic coalescence process. Despite its simplicity, it has profound mathematical properties that make it a cornerstone in population genetics. As noted in the introduction (Section \ref{Sec1}), the block counting process of Kingman's coalescent corresponds to a subcritical cooperative BPI process with parameters $c > 0$ and $d = \rho = b = 0$. In particular, it satisfies $\mathbf{m}_{\rm int}<0$, $\mathbf{m}_{\rm br}<\infty$,  \eqref{Assump3TWO} and \eqref{Assump4}, making the main results in Section \ref{MainResultSec} applicable.

The Ancestral Selection Graph (ASG), introduced in \cite{KRONE1997, Neuhauser1997}, extends Kingman's coalescent by incorporating the effects of natural selection and mutation. Mutation modifies the genetic type of an ancestral line, potentially leading to its loss when tracing lineages backward in time. Selection, on the other hand, provides an evolutionary advantage to fitter individuals. Unlike the tree-like structure of Kingman's coalescent, the ASG takes the form of a graph due to uncertainty introduced when tracing lineages backward without knowledge of individual types. Specifically, it is unclear whether an offspring results from standard reproduction or involves a fit individual. To capture this ambiguity, the ASG allows lineages to split into two potential paths. The block counting process of the ASG is a subcritical cooperative BPI process with parameters $c,d > 0$, $b = 0$, $\pi_{1}= \rho>0$ and $\pi_{i}=0$, for all $i \geq 2$.
 Here, $d$  accounts for lineage loss due to mutations, while 
$\rho$ captures lineage splitting caused by selection. Clearly, $\mathbf{m}_{\rm int}<0$, $\mathbf{m}_{\rm br}<\infty$, and \eqref{Assump3TWO} are satisfied, ensuring that the block counting process of the ASG comes down from infinity at the same speed as Kingman's coalescent (Theorem \ref{Main1}); specifically, $\mathbf{m}_{\rm int} = -c$. Furthermore, \eqref{Assump4} is satisfied, and Theorem \ref{Main3} applies with $\boldsymbol{\sigma}_{{\rm int}}^{2} = c$. These results were established for this specific model in \cite{Hanson2020} (unfortunately, unpublished).

The dynamics of the block counting process in the ASG align with the so-called $(1, \rho, c, d)$-braco-process \cite{griffiths1997}, which tracks the number of particles in a branching-coalescing system with the following rules: particles split into two at rate $\rho$, die at rate 
$d$, and any ordered pair of particles coalesces into one at rate 
$c$. Moreover, as highlighted in \cite{Hanson2020}, the block counting process of the ASG bears similarities to the Ancestral Recombination Graph (ARG) \cite{Athreya2005, Alkemper2007}, a BPI process characterised by parameters $c > 0$, $b = d = 0$, $\pi_{1}= \rho>0$ and $\pi_{i}=0$, for all $i \geq 2$.

The Ancestral Selection/Efficiency Graph (ASEG), introduced in \cite{GONZALEZ2020}, generalises the ASG by incorporating the concept of efficiency, an ecological factor related to resource consumption strategies. Efficient individuals require fewer resources to produce progeny, giving them an advantage when resources are scarce, while inefficient individuals require more. The vertex counting process of the ASEG, as studied in \cite{GONZALEZ2020}, corresponds to a BPI process with parameters  $c >0$, $b \in [0,c]$, $d = 0$, $\pi_{1}= \rho \geq 0$ and $\pi_{i}=0$, for all $i \geq 2$. Here, 
$b$ represents the efficiency parameter, 
$\rho$ reflects selection, and there are no mutations $d=0$. In \cite{GONZALEZ2020}, the authors adopt the parametrisation $c = 1/2$, $b = \kappa/2$, and $\rho = \alpha$, for $\kappa \in [0, 1]$ and $\alpha \geq 0$. When $b \in [0,c)$, the process is subcritical cooperative (i.e.\ $\mathbf{m}_{\rm int}<0$), satisfying $\mathbf{m}_{\rm br}<\infty$ and \eqref{Assump3TWO}. Thus, the vertex counting process of the ASEG comes down from infinity at the same speed (up to a constant) as Kingman's coalescent (Theorem \ref{Main1}), with $\mathbf{m}_{\rm int} = -c+b$. Additionally, \eqref{Assump4} is satisfied, and Theorem \ref{Main3} applies with $\boldsymbol{\sigma}_{{\rm int}}^{2} = c+b$.

\subsection{Comments about our results and organization of the rest of the paper}\label{comments}

The proof of Theorem \ref{Main1} builds upon Theorem 4.5 in \cite{Bansaye2019} (see Section \ref{ProofofThemMain} for further details). However, addressing these results requires overcoming several technical challenges. Specifically, precisely determining the speed at which the processes come down from infinity. A central aspect of our approach is the representation of BPI processes as solutions to stochastic differential equations (SDEs) with jumps (see Section \ref{IntegralR}), allowing the effective application of martingale methods and stochastic calculus.

A complementary approach on the phenomenon of coming down from infinity for BPI processes is provided in Theorem 1.14 in \cite{Berzunza2020}. This result offers a refined perspective using a moment duality approach. However, Theorem 1.14 in \cite{Berzunza2020} does not address the rate at which subcritical (or critical) cooperative BPI processes come down from infinity, leaving the question open within this framework. In the subcritical cooperative regime, the conditions in \cite{Berzunza2020} hold under the assumption that $\mathbf{m}_{\rm br} < \infty$, which implies that the process does not explode (see Example 2.1 in \cite{Berzunza2020}). Additionally, in this case, the integral $\mathcal{I}_{d,\rho}^{c,b}(\theta; 1)$ from \cite{Berzunza2020} is always finite, which indicates that the associated BPI process comes down from infinity (see the comment following Theorem 1.14 in \cite{Berzunza2020}).

In the critical cooperative regime, the situation is more delicate. There are situations where the test integral in Theorem 1.14 of \cite{Berzunza2020} is unable to determine whether the associated BPI process comes down from infinity. For example, in the case where $\mathbf{m}_{\rm int} =0$ and $\mathbf{m}_{\rm br} < \boldsymbol{\sigma}^{2}_{\rm int}/2 < \infty$, the integral $\mathcal{I}_{d,\rho}^{c,b}(\theta; 1)$ from \cite{Berzunza2020} is infinite (since we do not deal with the critical cooperative case here, details are left to the interested reader). Furthermore, the results from \cite{Bansaye2019}, specifically Proposition 4.4 and Theorem 4.5, are not straightforward to apply. Their use requires verifying non-trivial conditions that depend on the choice of a suitable function, denoted as $F$. While $F(z) \coloneqq \log(1+z)$ works for the subcritical cooperative regime (see Section \ref{ProofofThemMain}), it is unclear if a similar function exists or can be found for the critical cooperative regime. Thus, we leave the question of the speed at which critical cooperative BPI processes come down from infinity as an open problem for future work. 

Let us emphasise that the choice of $F(z) \coloneqq \log(1+z)$ in the subcritical cooperative regime, used to determine the speed of coming down from infinity, is not necessarily unique.  Other functions may yield the same conclusion. For instance, one could consider $F(z) \coloneqq (1+z)^{a}$; however, verifying the conditions required in \cite{Bansaye2019} appears to be more delicate in this case. Moreover, it remains unclear whether alternative choices of  functions
 could be exploited to weaken the moment assumptions in the subcritical cooperative regime. We leave this question as an open problem. For related examples where different choices of functions are employed to study coming down from infinity in other stochastic processes, see Section 4.2 of \cite{Bansaye2019}.

We now focus on  Theorems \ref{PropC2} and \ref{Main3}. As previously noted,  in  assumption \eqref{Assump4}, the condition $\vartheta>e^{288}$ can be relaxed to $\vartheta>1$. However, this adjustment comes at a cost: under our methodology, the  $\mathbb{L}_p$ convergence in Theorem \ref{PropC2} will only hold   for $p\in (0, \log(\vartheta)/144]$. Furthermore, if  $\vartheta\le e^{288}$, our approach is insufficient to establish the conclusions of Theorem \ref{Main3}.  Assumption \eqref{Assump4} is thus critical to the robustness of our result  in Theorem \ref{Main3}, and any relaxation of its conditions requires a thorough examination of its implications. Nonetheless, we conjecture that Theorems \ref{PropC2} and \ref{Main3} may remain valid under weaker assumptions. This leaves open the possibility of further investigation.  In particular, we suspect that imposing moment conditions on the branching parameters may not be necessary, as the dominant contributions arise from interactions whose rate of occurrence is quadratic.

The proof of Theorem \ref{PropC2} (see Section \ref{ProofofPropC2}) relies on Lemma \ref{lemma3}, which is a key improvement to the estimate provided in Lemma 4.6 of \cite{Bansaye2019} (see also Theorem 3.2 of \cite{Bansaye2019}). This refinement enables the application of a supermartingale inequality (see Lemma \ref{lemmaA1} in Appendix \ref{Apendice}), which resembles the approach used in the proof of Theorem 2 in \cite{Bere2010} where the $\mathbb{L}_p$ convergence for the speed of coming down from infinity for $\Lambda$ coalescents was studied. (This approach may be of independent interest, as it has the potential to be applied to a broader setting, such as that considered in \cite{Bansaye2019}.)

A crucial component of the proof of Theorem \ref{Main3} is the result in Theorem \ref{PropC2} for $p=2$ (see Section \ref{ProofofMain3} for further details). The proof of Theorem \ref{Main3} builds upon the robust framework developed in \cite{Vlada20152, Vlada2015}, where the fluctuations for the speed of coming down from infinity for $\Lambda$ coalescents are studied, with key modifications to account for the branching and cooperative dynamics of the BPI process. 

\section{Proofs of the main results}

Before we proceed with the proofs of our main results, we fix some notation.  For two real-valued functions $f, g:[0, \infty) \rightarrow \mathbb{R}$, we write $f(x) \sim g(x)$, as $x \rightarrow y$, if
$\lim_{x \rightarrow y} f(x)/g(x) = 1$, where $y \in [0,\infty]$.
We denote by $\overline{B}(x, \varepsilon) \coloneqq \{y \in \mathbb{R}_{+}: |x-y| \leq \varepsilon \}$ the Euclidean closed ball centred at $x \in \mathbb{R}_{+}$ with radius $\varepsilon >0$. More generally, we denote by $\overline{B}_{d}(x, \varepsilon) \coloneqq \{y \in \mathbb{R}_{+}: d(x,y) \leq \varepsilon \}$ the closed ball centred at $x \in \mathbb{R}_{+}$ with radius $\varepsilon >0$ associated with the application $d: \mathbb{R}_{+} \times  \mathbb{R}_{+} \rightarrow \mathbb{R}_{+}$. Let $\Delta X_{s} = X_{s} - X_{s-}$ denote the jump at time $s \geq 0$ of a c\`adl\`ag process $(X_{t})_{t \geq 0}$. 

\subsection{Integral representation of BPI processes}  \label{IntegralR}

In this section, we provide a representation of  a BPI process as the unique strong solution of a stochastic differential equation. Let $(\Omega, \mathcal{F}, (\mathcal{F}_{t})_{t \geq 0}, \mathbb{P})$ be a filtered probability space satisfying the usual conditions. Let $\nu_{\rm br}$ and $\nu_{\rm int}$ be two $\sigma$-finite measures on $\mathbb{N}_{-1}:=\{-1\}\cup \mathbb{N}_0$ given by
\begin{align} \label{Measuresbrint}
\nu_{\rm br}({\rm d} x) = d \delta_{-1}({\rm d} x) + \sum_{i = 1}^{\infty} \pi_{i} \delta_{i}({\rm d} x) \qquad \text{and} \qquad \nu_{\rm int}({\rm d} x) = 2c \delta_{-1}({\rm d} x) + \sum_{i = 1}^{\infty} 2b_{i} \delta_{i}({\rm d} x),
\end{align}
where $\delta_u$ denotes the Dirac measure at $u$.

Set $\Delta \coloneqq  \{ (i,j) \in \mathbb{N} \times \mathbb{N}: 1 \leq i < j \}$. Let $N_{\rm br}$ and $N_{\rm int}$ be two independent $(\mathcal{F}_{t})_{t \geq 0}$-Poisson point measures on $\mathbb{R}_{+} \times \mathbb{N} \times \mathbb{N}_{-1}$ and $\mathbb{R}_{+} \times \Delta \times \mathbb{N}_{-1}$ with intensity measures ${\rm d} s \otimes \sum_{i \geq 1}\delta_{i}({\rm d} y) \otimes \nu_{\rm br}({\rm d} x)$ and ${\rm d} s  \otimes \sum_{(i,j) \in \Delta}\delta_{(i,j)}({\rm d} y) \otimes \nu_{\rm int}({\rm d} x)$, respectively; where ${\rm d} s$ denotes the Lebesgue measure on $\mathbb{R}_{+}$. For $k \geq 2$ (an integer), let us denote $\Delta_{k} \coloneqq \{ (i,j) \in \Delta: 1 \leq i <j\leq k\}$ and set $\Delta_{1} = \Delta_{0} = \emptyset$. Observe that $\nu_{\rm br}$ and $\nu_{\rm int}$ satisfy 
\begin{align} \label{LevyM1}
\int_{\mathbb{N}_{-1}} (1 \wedge x^{2}) \nu_{\rm br}({\rm d} x) < \infty \qquad \text{and} \qquad \int_{\mathbb{N}_{-1}} (1 \wedge x^{2}) \nu_{\rm int}({\rm d} x) < \infty,
\end{align}
\noindent respectively (i.e.\ $\nu_{\rm br}$ and $\nu_{\rm int}$ are L\'evy measures).

Set $\mathcal{X} \coloneqq \{1,2\} \times (\mathbb{N} \cup \Delta ) \times \mathbb{N}_{-1}$, $\mathbb{N}_{0, \infty}:= \mathbb{N}_{0}\cup\{\infty\}$ and let
\begin{align} \label{eq2I}
N({\rm d} s, {\rm d} u, {\rm d} y, {\rm d} x) = \delta_{1}({\rm d} u) N_{\rm br}({\rm d} s,  {\rm d} y, {\rm d} x) + \delta_{2}({\rm d} u) N_{\rm int}({\rm d} s,  {\rm d} y, {\rm d} x)
\end{align}
\noindent be a Poison point measure on $\mathbb{R}_{+} \times \mathcal{X}$ with intensity measure
\begin{align} \label{InteMeasu}
{\rm d} s \otimes q({\rm d} u, {\rm d} y, {\rm d} x) = {\rm d} s \otimes \Big( \delta_{1}({\rm d} u) \otimes \sum_{i \geq 1}\delta_{i}({\rm d} y) \otimes \nu_{\rm br}({\rm d} x) +  \delta_{2}({\rm d} u) \otimes \sum_{(i,j) \in \Delta}\delta_{(i,j)}({\rm d} y) \otimes \nu_{\rm int}({\rm d} x) \Big).
\end{align}
\noindent Let $\tilde{N}$ denote the compensated Poisson point measure of $N$.

\begin{proposition} \label{Pro1}
For $z \in \mathbb{N}_{0}$, the stochastic differential equation
\begin{align} \label{SDE}
Z_{t} = z + \int_{0}^{t} \int_{\mathcal{X}} H(Z_{s-}, u, y, x) N({\rm d} s, {\rm d} u, {\rm d} y, {\rm d} x), \quad t \geq 0,
\end{align}
\noindent where 
\begin{align} \label{funcH}
H(z, u, y, x) = \mathbf{1}_{\{u=1\}} x \mathbf{1}_{\{y \in \mathbb{N}, 0 < y \leq z \}} + \mathbf{1}_{\{u=2\}} x \mathbf{1}_{\{y \in \Delta_{z}\}}, \quad  (z, u, y, x) \in \mathbb{N}_{0} \times \mathcal{X}, 
\end{align}

\noindent has a unique strong solution. The solution $(Z_{t})_{t \geq 0}$ is a $\mathbb{N}_{0, \infty}$-valued $(\mathcal{F}_{t})_{t \geq 0}$-adapted process with c\`adl\`ag paths such that $\mathbb{P}(Z_{0} = z)=1$ and it satisfies \eqref{SDE} up to time $\tau_{n} \coloneqq \inf\{t \geq 0: Z_{t} \geq n\}$ for all $n \in \mathbb{N}$, and $Z_{t} = \infty$ for all $t \geq \tau \coloneqq \lim_{n \rightarrow \infty} \tau_{n}$. Moreover, the solution is a BPI-process issued from $z$.
\end{proposition}

\begin{proof}
The result follows from Proposition 1 in \cite{Palau2018} by considering the functions $b \equiv 0$ (no drift term), $\sigma \equiv 0$ (no Gaussian term), $h \equiv 0$ (no compensated Poisson measure term), $M=N$, $U = \mathcal{X}$ and $g=H$. The conditions of Proposition 1 in \cite{Palau2018} are verified using \eqref{LevyM1}.
\end{proof}

We observe that the proof of Proposition \ref{Pro1} does not require any assumptions on the parameters of the BPI process. Henceforth, let $(Z_{t}^{(z)})_{t \geq 0}$ denote the BPI process issued from $z \in \mathbb{N}_{0}$, that is, the unique strong solution of the SDE \eqref{SDE}. Note that the SDE representation \eqref{SDE} of the BPI process bears some similarity to the construction of the standard Kingman coalescent, as explained in Section 2 of \cite{Vlada20152}. Recall that the block counting process of the Kingman coalescent is a special case of a BPI process, when $c>0$, $d=\rho=b=0$.

\begin{proposition} \label{Prop6}
Suppose that $\sum_{i \geq 1} i\pi_{i} < \infty$ and $\sum_{i \geq 1} i b_{i} < \infty$. Then, for $z_{1}, z_{2} \in \mathbb{N}$ such that $z_{1} \leq z_{2}$, we have that $\mathbb{P}(Z^{(z_{1})}_{t} \leq Z^{(z_{2})}_{t} \, \, \text{for all} \, \, t \geq 0) =1$. 
\end{proposition}

\begin{proof}
The proof is a simple adaptation of the argument used in the Proof of Theorem 5.5 in \cite{Zenghu2010}. Let $Y_{t} = Z^{(z_{1})}_{t} - Z^{(z_{2})}_{t}$, for $t \geq 0$. By Proposition \ref{Pro1},
\begin{align} \label{Inteq1}
Y_{t}= z_{1} -z_{2} + \int_{0}^{t} \int_{\mathcal{X}} ( H(Z^{(z_{1})}_{s-}, u, y, x) - H(Z^{(z_{2})}_{s-}, u, y, x)) N({\rm d} s, {\rm d} u, {\rm d} y, {\rm d} x).
\end{align}

Let $(a_{m})_{m \geq 0}$ be a decreasing sequence of positive real numbers such that $a_{0} =1$, $a_{m} >0$ for $m \in \mathbb{N}$, and $\lim_{m \rightarrow \infty} a_{m} = 0$. For each $m \in \mathbb{N}$, let $g_{m}: \mathbb{R} \rightarrow \mathbb{R}_{+}$ be a non-negative continuous function with support on $(a_{m}, a_{m-1})$ and satisfying $\int_{a_{m}}^{a_{m-1}} g_{m}(z) {\rm d} z= 1$. For each $m \in \mathbb{N}$, we define the non-negative and twice continuously differentiable function
\begin{align}
\phi_{m}(z) = \int_{0}^{z} \int_{0}^{y} g_{m}(x) {\rm d} x {\rm d} y, \qquad z \in \mathbb{R}. 
\end{align}
\noindent Clearly, $\phi_{m}(z) \rightarrow z^{+} = 0 \vee z$, for $z \in \mathbb{R}$, non-decreasingly, as $m \rightarrow \infty$. Since $\phi_{m}(z) = 0$, for $z \leq 0$, and $z \mapsto z + H(z, u,y,x)$ is non-decreasing, for $(u,y,x) \in \mathcal{X}$, if $Y_{s-} \leq 0$, we have 
\begin{align}
Y_{s-} + H(Z^{(z_{1})}_{s-}, u, y, x) - H(Z^{(z_{2})}_{s-}, u, y, x) \leq 0
\end{align}
 and hence
\begin{align} \label{Inteq2}
G_{s}(\phi_{m},Y) \coloneqq \phi_{m}(Y_{s-}+  H(Z^{(z_{1})}_{s-}, u, y, x) - H(Z^{(z_{2})}_{s-}, u, y, x)) - \phi_{m}(Y_{s-}) =0, \quad s \geq 0.
\end{align}
\noindent Note that $Y_{0} = z_{1} - z_{2} \leq 0$ and $\phi_{m}(Y_{0}) = 0$.  For $n \in \mathbb{N}$, define $\tau_{n} = \inf\{t \geq 0: Z^{(z_{1})}_{t} \geq n \, \, \text{or} \, \,  Z^{(z_{2})}_{t} \geq n \}$. Then, by \eqref{Inteq1}, It\^o's formula (Theorem 5.1  in Chapter 2 of \cite{Ikeda1981}) and \eqref{Inteq2},
\begin{align} \label{Inteq3}
\phi_{m}(Y_{t \wedge \tau_{n}}) & = \int_{0}^{t \wedge \tau_{n}} \int_{\mathcal{X}} G_{s}(\phi_{m}, Y) N({\rm d} s, {\rm d} u, {\rm d} y, {\rm d} x) \nonumber \\
&  = \int_{0}^{t \wedge \tau_{n}} \int_{\mathcal{X}} G_{s}(\phi_{m}, Y) \mathbf{1}_{\{Y_{s-} >0 \}} {\rm d} s q({\rm d} u, {\rm d} y, {\rm d} x)   + M_{t}^{(n,m)},
\end{align}
\noindent where $(M_{t}^{(n,m)})_{t \geq 0}$ is martingale. Note that $|\phi_{m}^{\prime}(z)| \leq 1$, for $z \in \mathbb{R}$, which implies that  $\phi_{m}$ is (globally) Lipschitz. Since we have assumed that $\sum_{i \geq 1} i\pi_{i} < \infty$ and $\sum_{i \geq 1} i b_{i} < \infty$, it follows from \eqref{InteMeasu} and \eqref{funcH}, for any $s \leq \tau_{n}$,
\begin{align} \label{Inteq4}
& \int_{\mathcal{X}} G_{s}(\phi_{m}, Y)  \mathbf{1}_{\{Y_{s-} >0 \}}  q({\rm d} u, {\rm d} y, {\rm d} x) \leq \Big( d + (2n-1)c + \sum_{i \geq 1}i (\pi_{i}+(2n-1)b_{i}) \Big) Y_{s-}^{+}.
\end{align}
\noindent Then, from \eqref{Inteq3}, \eqref{Inteq4} and the dominated convergence theorem (letting $m \rightarrow \infty)$, we deduce that
\begin{align}
\mathbb{E}[Y_{t \wedge \tau_{n}}^{+}] \leq  \Big( d + (2n-1)c + \sum_{i \geq 1}i (\pi_{i}+(2n-1)b_{i}) \Big)  \int_{0}^{t} \mathbb{E}[Y_{s \wedge \tau_{n}}^{+} ] {\rm d} s. 
\end{align}
\noindent Then, $\mathbb{E}[Y_{t \wedge \tau_{n}}^{+}] = 0$, for all $t \geq 0$. On the other hand, since $(Z^{(z_{1})}_{t})_{t \geq 0}$ and $(Z^{(z_{2})}_{t})_{t \geq 0}$ do not explode in finite time and have c\`adl\`ag paths, we deduce that $\tau_{n} \rightarrow \infty$, as $n \rightarrow \infty$. Therefore, our claim follows. 
\end{proof}

\begin{lemma} \label{lemma4}
Suppose that $\sum_{i \geq 1} i\pi_{i} < \infty$ and $\sum_{i \geq 1} i b_{i} < \infty$.  Then, the BPI process is stochastically monotone, that is, for all $z_{1}, z_{2} \in \mathbb{N}_{0}$ such that $z_{1} \leq z_{2}$, $\mathbb{P}(Z_{t}^{(z_{1})} \geq x) \leq \mathbb{P}(Z_{t}^{(z_{2})} \geq x)$, for all $t \geq 0$ and $x \in \mathbb{R}$.
\end{lemma}
\begin{proof}
It follows from Proposition \ref{Prop6}.
\end{proof}

\subsection{Proof of Theorem \ref{Main1}} \label{ProofofThemMain}

In this section we prove Theorem \ref{Main1} using  the framework developed in Section 4 of \cite{Bansaye2019} (with $F(z) \coloneqq \log(1+z)$, for $z \in (-1, \infty)$). Recall that $\mathcal{X} \coloneqq \{1,2\} \times (\mathbb{N} \cup \Delta) \times \mathbb{N}_{-1}$. The next lemma verifies part of Assumptions 4.1 and 4.3 in \cite{Bansaye2019}. We have removed the subscript $F$ from our notation and will use $\varphi$ to denote the function $b_{F}$ in \cite{Bansaye2019}.

\begin{lemma} \label{lemma1}
Suppose that $\mathbf{m}_{\rm int}<0$ and $\mathbf{m}_{\rm br}<\infty$. We have the following:
\begin{enumerate}[label=(\roman*)]
\item For any $z \in \mathbb{N}_{0}$, $\displaystyle \int_{\mathcal{X}} \Big|\log \Big(1+\frac{H(z, u, y, x)}{z+1} \Big) \Big| q({\rm d} u, {\rm d} y, {\rm d} x) < \infty$. \label{Prop1}

\item Define the function $\displaystyle h(z) \coloneqq \int_{\mathcal{X}} \log \Big(1+\frac{H(z, u, y, x)}{z+1} \Big) q({\rm d} u, {\rm d} y, {\rm d} x)$, for $z \in \mathbb{N}_{0}$. Then, $h$ can be extended to $[0, \infty)$ such that is locally bounded on $[0, \infty)$ and locally Lipschitz on $(0, \infty)$.   \label{Prop2}

\item There exists $z_0 >0$ such that the function $\varphi(z) \coloneqq (z+1) h(z)$ is strictly negative on $(z_0, \infty)$.  \label{Prop3}

\item Define the function $\displaystyle \psi(z) \coloneqq  e^{-z}\varphi(e^{z}-1)$, for $z \in [0,\infty)$. Then, $\psi$ is $(0, \kappa)$ nonexpansive on $(1,\infty)$ for some $\kappa \geq 0$, i.e. for all $z_{1} > z_{2} > 1$, $\psi(z_{1}) \leq \psi(z_{2}) + \kappa$.  \label{Prop4}
\end{enumerate}
\end{lemma}

\begin{proof}
Recall that, for $x \geq 0$, 
\begin{align} \label{logIne1}
\frac{x}{x+1} \leq \ln (1+x) \leq x.
\end{align}

First, we prove \ref{Prop1}. If $z=0$, then $H(0, u, y, x)=0$, for all $u=1,2$, $y \in \mathbb{N} \cup \Delta$ and $x \in \mathbb{N}_{-1}$ and thus, \ref{Prop1} holds. For $z \in \mathbb{N}$, \ref{Prop1} follows from \eqref{InteMeasu}, \eqref{funcH}, \eqref{logIne1} and our assumptions $\mathbf{m}_{\rm int}<0$ and $\mathbf{m}_{\rm br}<\infty$, that is,
\begin{align}  \label{eq8s}
\int_{\mathcal{X}} \Big|\log \Big(1+\frac{H(z, u, y, x)}{z+1} \Big)\Big| q({\rm d} u, {\rm d} y, {\rm d} x) & = z \int_{\mathbb{N}_{-1}} \Big|\log \Big(1+ \frac{x}{z+1} \Big) \Big| \nu_{\rm br}({\rm d} x) \nonumber \\
& \quad \quad  \quad + \frac{z(z-1)}{2} \int_{\mathbb{N}_{-1}} \Big|\log \Big( 1+\frac{x}{z+1} \Big)\Big| \nu_{\rm int}({\rm d} x) \nonumber \\
& \leq d + \sum_{i \geq 1} i \pi_{i} + zc + z \sum_{i \geq 1} i b_{i} < \infty.
\end{align}

Next, we prove \ref{Prop2}. Note that by \eqref{InteMeasu} and \eqref{funcH}, 
\begin{align} \label{eq9s}
h(z) &  = z \int_{\mathbb{N}_{-1}} \log \Big( 1+ \frac{x}{z+1} \Big) \nu_{\rm br}({\rm d} x) + \frac{z(z-1)}{2} \int_{\mathbb{N}_{-1}} \log \Big(1+ \frac{x}{z+1} \Big) \nu_{\rm int}({\rm d} x) = h_{\rm br}(z)+h_{\rm int}(z),
\end{align}
\noindent for $z \in \mathbb{N}_{0}$, where 
\begin{align} \label{branchingh}
h_{\rm br}(z) := - z d \log \Big( 1 + \frac{1}{z} \Big) + z\sum_{i \geq 1} \pi_{i} \log \Big( 1+ \frac{i}{z+1} \Big)
\end{align}
\noindent and 
\begin{align} \label{interationh}
h_{\rm int}(z) := - z(z-1) c \log\Big( 1 + \frac{1}{z} \Big)  + z(z-1)\sum_{i \geq 1} b_{i} \log \Big( 1+ \frac{i}{z+1} \Big).
\end{align}
\noindent Clearly, $h$ (and thus, $h_{\rm br}$ and $h_{\rm int}$) can be extended to $[0,\infty)$. In particular, \eqref{eq8s} remains valid and it follows that $h$ is locally bounded on $[0,\infty)$. On the other hand, note that, for $0< x \leq y$, \eqref{logIne1} implies that
\begin{align} \label{logIne}
|\log (x) - \log(y)| = \log \left( \frac{y}{x} \right) = \log \left(1 + \left( \frac{y}{x} - 1 \right) \right) \leq \frac{y-x}{x}.
\end{align}
\noindent Then, by \eqref{logIne} and our assumptions that $\mathbf{m}_{\rm int}<0$ and $\mathbf{m}_{\rm br}<\infty$, it is not difficult to deduce that $h$ is locally Lipschitz on $(0, \infty)$, which concludes with the proof of \ref{Prop2}. 

We now prove \ref{Prop3}. By \eqref{logIne1} and  since $\mathbf{m}_{\rm br}<\infty$, $h_{\rm br}$ is a bounded function on $[0, \infty)$. On the other hand, by \eqref{logIne1}, for $z \in [0, \infty)$, 
\begin{align}
h_{\rm int}(z) & \leq - z(z-1) c \log\Big( 1 + \frac{1}{z} \Big) + (z-1) \sum_{i \geq 1} i b_{i} \nonumber \\
& = - (z-1) c \left( \log\Big( 1 + \frac{1}{z} \Big) -1 \right) + (z-1) \mathbf{m}_{\rm int} \nonumber \\
& \leq \frac{z-1}{z+1} c + (z-1) \mathbf{m}_{\rm int}.
\end{align}
\noindent In particular, since  $\mathbf{m}_{\rm int}<0$, $h_{\rm int}(z) \rightarrow - \infty$, as $z \rightarrow \infty$, which implies \ref{Prop3}. 

Finally, we prove \ref{Prop4}. For $z \in [0,\infty)$, we have that $h(z) = -c z + h_{\rm int}(z) + z(z-1) c \log( 1 + 1/z) + \hat{h}(z)$, where $\hat{h}(z) = h_{\rm br}(z) - z(z-1) c \log( 1 + 1/z) + c z$. For $x, y >0$, we have that
\begin{align} \label{AnotherLogIne}
\left| \frac{1}{x}- \frac{1}{y}  \right| = \frac{|x-y|}{xy}. 
\end{align}
\noindent Then, by the triangle inequality and \eqref{AnotherLogIne}, we have that, for $z \in (1, \infty)$,
\begin{align}  \label{eq12s}
\Big | - z(z-1) c \log( 1 + 1/z) + c z   \Big | & \leq z(z-1)\Big | c \int_{0}^{-1} \Big(\frac{1}{z+1+y} -  \frac{1}{z-1}\Big) {\rm d}y   \Big | \leq 2c. 
\end{align}
\noindent Recall that $h_{\rm br}$ is bounded on $[0, \infty)$ and thus, \eqref{eq12s} implies that $\hat{h}$ is a bounded function on $(1, \infty)$. Moreover, $z \mapsto -c z + h_{\rm int}(z) + z(z-1) c \log( 1 + 1/z)$ is a decreasing function on $(1, \infty)$. Then, \ref{Prop4} follows with $\kappa = 2||\hat{h}||_{\infty}$. 
\end{proof}

Let $z_0 > 1$ be such that \ref{Prop3} in Lemma \ref{lemma1} is satisfied. Recall the definition of the function $\varphi$ in Lemma \ref{lemma1} \ref{Prop3}. By Lemma \ref{lemma1} \ref{Prop1}-\ref{Prop2}, $\varphi$ is well-defined and locally Lipschitz on $D \coloneqq (z_0, \infty)$. We introduce the flow $\phi$ associated to $\varphi$ and defined for $z \in D$ as the unique solution of
\begin{align} \label{Flow}
 \quad \frac{\partial}{\partial t} \phi(z,t) = \varphi(\phi(z,t)), \qquad \textrm{with}\quad \phi(z,0) = z,
\end{align}
\noindent for $t \in [0, T(z))$, where $T(z) \in (0, \infty]$ is the maximal time until which the solution exists and belongs to $D$. Observe that $z \in D \rightarrow \phi(z,t)$ is increasing where it is well-defined. Then $T(z)$ is increasing and its limit when $z \uparrow \infty$ is denoted by $T(\infty)$ and belongs to $(0, \infty]$. Moreover, the flow starting from infinity is well-defined by a monotone limit,
\begin{align} \label{Def1Speed}
v_{t} \coloneqq \phi(\infty,t) = \lim_{z \rightarrow \infty} \phi(z,t),
\end{align}
\noindent for any $t \in [0, T(\infty))$. By Lemma \ref{lemma1} \ref{Prop3}, for $z \in D$, $\varphi(z) <0$ and for any $t < T(z)$, 
\begin{align}
\int_{z}^{\phi(z,t)} \frac{{\rm d} x}{\varphi(x)} = t.
\end{align}

\begin{lemma} \label{lemma2}
Suppose that $\mathbf{m}_{\rm int}<0$ and $\mathbf{m}_{\rm br}<\infty$. Then, $\varphi(z) \sim z^{2} \mathbf{m}_{\rm int}$, as $z \rightarrow \infty$. In particular, for any $z \in D$, 
\begin{align}
-\int_{z}^{\infty} \frac{{\rm d} x}{\varphi(x)} < \infty.
\end{align}
\end{lemma}
\begin{proof}
The first claim follows from the definition of $\varphi$,  \eqref{eq9s} and the dominated convergence theorem. The second claim follows clearly from the first claim. 
\end{proof}

By Lemma \ref{lemma2}, we have that $v_{t}$ satisfies 
\begin{align} \label{Def2Speed}
-\int_{v_{t}}^{\infty} \frac{{\rm d} x}{\varphi(x)} = t,
\end{align}
\noindent and in particular,
\begin{align}
v_{t} = \inf \Big \{ z \in D:  \int_{z}^{\infty} \frac{{\rm d} x}{- \varphi(x)} < t \Big\} < \infty,
\end{align}
\noindent for any $t \in (0, T(\infty))$. Following Section 4 in \cite{Bansaye2019}, one says that the dynamical system instantaneously comes down from infinity. The function $t \in [0, T(\infty)) \rightarrow v_{t} \in \overline{\mathbb{R}}$ is continuous when $\overline{\mathbb{R}}$ is endowed with the distance $\overline{d}(x,y) = |e^{-x}-e^{-y}|$, for $x,y \in  \overline{\mathbb{R}}$. 

\begin{lemma} \label{lemma6}
Suppose that $\mathbf{m}_{\rm int}<0$ and $\mathbf{m}_{\rm br}<\infty$. We have the following:
\begin{enumerate}[label=(\roman*)]
\item $v_{t} >0$, for all $t \in (0, T(\infty))$ and $\lim_{t \downarrow 0} v_{t} = \infty$.   \label{PropB1}

\item the map $t\mapsto v_t$ is differentiable on $(0, T(\infty))$ and $\frac{\rm d}{\rm d t} v_{t} = \varphi(v_{t})$,  for $t \in (0, T(\infty))$. In particular, $t\mapsto v_t$ is decreasing.  \label{PropB2}

\item $v_{t} \sim -\frac{1}{t\mathbf{m}_{\rm int}}$, as $t \downarrow 0$.  \label{PropB3}
\end{enumerate}
\end{lemma}

\begin{proof}
By \eqref{Def1Speed}, $v_{t} \in D$ and thus $v_{t} >0$, for $t \in (0, T(\infty))$. By Lemma \ref{lemma1} \ref{Prop3}, $-\varphi(\cdot)$ is strictly positive on $D$. Then, Lemma \ref{lemma2} implies that $z \mapsto g(z) \coloneqq -\int_{z}^{\infty}\frac{{\rm d} x}{\varphi(x)}$ maps $D=(z_0, \infty)$ bijectively to $(0, -\int_{z_0}^{\infty}\frac{{\rm d} x}{\varphi(x)})$. Moreover, $g$ is decreasing on $(z_0, \infty)$. Then, it should be clear that $\lim_{t \downarrow 0} v_{t} = \infty$, which proves \ref{PropB1}. \ref{PropB2} follows by \eqref{Def2Speed} and the fundamental theorem of calculus; recall that for $z \in D$, $\varphi(z) <0$. From \eqref{Def2Speed}, we see that $t v_{t} = v_{t} \int_{v_{t}}^{\infty} \frac{{\rm d} x}{-\varphi(x)}$, for $t \in (0, T(\infty))$. Then,  L'H\^opital's rule and Lemma \ref{lemma2} imply \ref{PropB3}, that is,
\begin{align}
\lim_{t \downarrow 0} - \mathbf{m}_{\rm int} t v_{t} = - \mathbf{m}_{\rm int} \lim_{z \rightarrow \infty} z \int_{z}^{\infty} \frac{{\rm d} x}{-\varphi(x)} = - \mathbf{m}_{\rm int} \lim_{z \rightarrow \infty} \frac{\frac{1}{\varphi(z)}}{- \frac{1}{z^{2}}} =1.
\end{align}
This completes the proof.
\end{proof}

The following lemma shows that condition (21) in Assumption 4.3 of \cite{Bansaye2019} is satisfied. 

\begin{lemma} \label{lemma1TWO}
Suppose that $\mathbf{m}_{\rm int}<0, \mathbf{m}_{\rm br}<\infty$ and \eqref{Assump3TWO} hold. Define the function 
\begin{align} \label{eq5TWO}
 V(z) \coloneqq \int_{\mathcal{X}} \Big(\log \Big(1+\frac{H(z, u, y, x)}{z+1} \Big) \Big)^{2} q({\rm d} u, {\rm d} y, {\rm d} x),
\end{align}
\noindent for $z \in \mathbb{N}_{0}$. Then, for all $t \in (0, \infty)$ and any $\varepsilon >0$,
\begin{align} \label{eq1TWO}
\int_{0}^{t} \widehat{V}_{\varepsilon}(z_0, s) {\rm d} s < \infty,
 \end{align}
\noindent where $\widehat{V}_{\varepsilon}(z_0, t) = \sup_{z \in \mathbb{N}_{0} \cap \mathfrak{D}_{\varepsilon}(z_0, t)} \varepsilon^{-2} V(z)$, with $\mathfrak{D}_{\varepsilon}(z_0, t) \coloneqq \{ z \in (z_0, \infty): \log(1+z) \leq \log(1+v_{t}) + \varepsilon \}$. 
\end{lemma}

\begin{proof}
If $z=0$, then $H(0, u, y, x)=0$, for $u=1,2$, $y \in \mathbb{N} \cup \Delta$ and $x \in \mathbb{N}_{-1}$ and thus, $V(0) = 0$. For $z \in (0, \infty)$, 
\begin{align}  \label{eq2TWO}
V(z) = z \int_{\mathbb{N}_{-1}} \Big(\log \Big(1+ \frac{x}{z+1} \Big) \Big)^{2} \nu_{\rm br}({\rm d} x)  + \frac{z(z-1)}{2} \int_{\mathbb{N}_{-1}} \Big(\log \Big( 1+\frac{x}{z+1} \Big)\Big)^{2} \nu_{\rm int}({\rm d} x).
\end{align}
\noindent  On the one hand, observe that $\log(1+x) \leq \sqrt{2x}$, $x \geq 0$. Then, it follows from \eqref{Measuresbrint} that, for $z \in (0, \infty)$, 
\begin{align}  \label{eq3TWO}
z \int_{\mathbb{N}_{-1}} \Big(\log \Big(1+ \frac{x}{z+1} \Big) \Big)^{2} \nu_{\rm br}({\rm d} x) & = zd \Big(\log \Big(1+ \frac{1}{z} \Big) \Big)^{2} + z\sum_{i \geq 1} \pi_{i} \Big(\log \Big(1+ \frac{i}{z+1} \Big) \Big)^{2} \nonumber \\
& \leq  2 \Big(d +  \sum_{i \geq 1} i \pi_{i} \Big),
\end{align} 
\noindent which is finite from our assumptions.

On the other hand, observe that $\log(1+x) \leq \sqrt{2} x^{\alpha/2}$, $x \geq 0$. Then, it follows from the previous inequality, \eqref{Measuresbrint} and \eqref{Assump3TWO} that, for $z \in (0, \infty)$, 
\begin{align}  
\frac{z(z-1)}{2} \int_{\mathbb{N}_{-1}} \Big(\log \Big(1+ \frac{x}{z+1} \Big) \Big)^{2} \nu_{\rm int}({\rm d} x) & = z(z-1)c \Big(\log \Big(1+ \frac{1}{z} \Big) \Big)^{2} + z(z-1)\sum_{i \geq 1} b_{i} \Big(\log \Big(1+ \frac{i}{z+1} \Big) \Big)^{2} \nonumber \\ 
& \leq  2  \Big(c+\sum_{i\geq 1} i^{\alpha}b_{i} \Big)z^{2-\alpha}. \label{eq4TWO}
\end{align} 
\noindent By combining \eqref{eq2TWO}, \eqref{eq3TWO} and \eqref{eq4TWO}, we see that there exists a constant $C_{\alpha} >0$ such that $V(z) \leq C_{\alpha} (z^{2-\alpha}+1)$, for $z \in [0, \infty)$. Moreover, Lemma \ref{lemma6} \ref{PropB3} implies that for any $t>0$, there exists a constant $C_{t}>0$ such that $v_{s} \leq C_{t} s^{-1}$ for $s \in [0,t]$. Therefore, for any $\varepsilon >0$, there exists a constant $C_{\alpha, \varepsilon,t} >0$ such that for all $s \in [0,t]$ we have that $\widehat{V}_{\varepsilon}(z_0, s) \leq C_{\alpha, \varepsilon,t} (s^{-2+\alpha}+1)$.  This implies \eqref{eq1TWO} (recall that $\alpha \in (1,2]$ in \eqref{Assump3TWO}). 
\end{proof}

Note that the function $\widehat{V}_{\varepsilon}$ in Lemma \ref{lemma1TWO} corresponds to $\hat{V}_{F, \varepsilon}$ in Assumption 4.3 of \cite{Bansaye2019}.

\begin{remark} \label{remark3TWO}
If \eqref{Assump3TWO} holds with $\alpha = 2$, then $\boldsymbol{\sigma}^{2}_{{\rm int}} <\infty$. Since we have already assumed that $\mathbf{m}_{\rm br}<\infty$, it follows from \eqref{eq2TWO}, \eqref{eq3TWO} and \eqref{eq4TWO} (with $\alpha = 2$)  that the function $V$ in \eqref{eq5TWO} is bounded. In particular, the function $\widehat{V}_{ \varepsilon}$ is also bounded, which implies \eqref{eq1TWO}.

Furthermore, the proof shows that the result of Lemma \ref{lemma1TWO} holds true even without the condition that $\mathbf{m}_{\rm int}<0$, as long as \eqref{Assump3TWO} is satisfied.
\end{remark}

We have now all the ingredients to prove Theorem \ref{Main1}.

\begin{proof}[Proof of Theorem \ref{Main1}]
It follows by applying Proposition 4.4 (i) and Theorem 4.5 (i) in \cite{Bansaye2019} (with $a^{\prime} =-1$, $a= z_0$ and $F(z) \coloneqq \log(1+z)$, for $z \in (-1, \infty)$), Lemma \ref{lemma2} and Lemma 3 in Chapter 3, Section 16 of \cite{Billingsley1999} that the family $(Z^{(z)})_{z \in \mathbb{N}_{0}}$ of BPI process converges weakly, in $\mathbb{D}(\mathbb{R}_{+}, \overline{\mathbb{R}})$, as $z \rightarrow \infty$, towards $Z^{(\infty)} = (Z^{(\infty)}_{t})_{t \geq 0}$. Moreover,
\begin{align} \label{Pro2}
\mathbb{P}\Big(Z^{(\infty)}_{t} < \infty \, \, \text{for all} \, \, t >0\Big) = 1 \qquad \text{and} \qquad \lim_{t \downarrow 0}  \frac{Z_{t}^{(\infty)}}{v_{t}} = 1, \quad \mathbb{P}-\text{almost surely}.
\end{align}
\noindent The condition in Proposition 4.4 of \cite{Bansaye2019} is verified by using Lemma \ref{lemma4} (Recall that, under  the assumptions that $\mathbf{m}_{\rm int}<0$ and $\mathbf{m}_{\rm br}<\infty$, the BPI process does not explode in finite time, see Example 2.1 in \cite{Berzunza2020}). The conditions in Theorem 4.5 of \cite{Bansaye2019} are verified by using Lemma \ref{lemma1}, Lemma \ref{lemma1TWO} and Lemma \ref{lemma2}. Finally, \eqref{Pro2} and 
Lemma \ref{lemma6} imply that $\lim_{t \downarrow 0}  t Z_{t}^{(\infty)} = (-\mathbf{m}_{\rm int})^{-1}$, $\mathbb{P}$-almost surely, which concludes our proof. 
\end{proof}

\begin{remark}
At the time of writing, we noticed a minor error in the proof of Proposition 2.2 in \cite{Bansaye2019}. This proposition is used to prove Theorem 3.2 in \cite{Bansaye2019}, which in turn is used to deduce Lemma 4.6 in \cite{Bansaye2019}. The latter is subsequently employed in the proofs of Proposition 4.4 and Theorem 4.5 in \cite{Bansaye2019}.

The error in the proof of Proposition 2.2 in \cite{Bansaye2019} lies in the set inequality for $B_{\eta}$ on page 2383. Specifically, the indicator $\mathbf{1}_{\{ S_{s-} \leq \varepsilon \}}$ is incorrectly placed inside the integrals. To match the expression at the top of page 2383, it should remain outside the integrals as $\mathbf{1}_{\{ S_{t-} \leq \varepsilon \}}$. While moving the indicator inside the integral is generally incorrect, first applying Markov's inequality yields the correct bound in (6). The absolute values resulting from the application of Markov's inequality then justify moving the indicator back inside. Consequently, the statement of Proposition 2.2 in \cite{Bansaye2019} remains correct, and the remainder of the proof stands.
\end{remark}

\subsection{Proof of Theorem \ref{PropC2}} \label{ProofofPropC2}

In this section, we prove Theorem \ref{PropC2}. The key ingredient is an improvement of the estimate obtained in Lemma 4.6 of \cite{Bansaye2019} (see also Theorem 3.2 of \cite{Bansaye2019}). To achieve this, we use a supermartingale inequality (see Lemma \ref{lemmaA1} in Appendix \ref{Apendice}) similar to that used in the proof of Theorem 2 of \cite{Bere2010}.

\begin{lemma} \label{lemma7}
Suppose that \eqref{Assump4} holds. For $\theta >1$, $z \in \mathbb{N}_{0}$ and $p\in \mathbb{N}$ such that $p\geq 2$, we have that
\begin{align} \label{Ineq1fact}
\int_{\mathcal{X}} \Big|\log \Big(1+\frac{H(z, u, y, x)}{z+1} \Big) \Big|^{p} q({\rm d} u, {\rm d} y, {\rm d} x) \leq \frac{p!}{(\ln \theta)^{p}} \Big(d \theta+c\theta+\sum_{i \geq 1} \theta^{i} \pi_{i}+\sum_{i \geq 1}\theta^{i}b_{i}\Big)
\end{align}
\noindent and 
\begin{align} \label{Ineq2fact}
\int_{\mathcal{X}} \Big|\log \Big(1+\frac{H(z, u, y, x)}{z+1} \Big) \Big|^{2p} q({\rm d} u, {\rm d} y, {\rm d} x) \leq \frac{2^{p} p!}{(\ln \theta)^{p}} \Big(d \theta+c\theta+\sum_{i \geq 1} \theta^{i} \pi_{i}+\sum_{i \geq 1}\theta^{i}b_{i}\Big).
\end{align}
\end{lemma}

\begin{proof}
If $z=0$, then $H(0, u, y, x)=0$, for $u=1,2$, $y \in \mathbb{N} \cup \Delta$ and $x \in \mathbb{N}_{-1}$ and thus, \eqref{Ineq1fact} is satisfied. By \eqref{InteMeasu} and \eqref{funcH}, 
\begin{align} \label{eq12}
\int_{\mathcal{X}} \Big|\log \Big(1+\frac{H(z, u, y, x)}{z+1} \Big) \Big|^{p} q({\rm d} u, {\rm d} y, {\rm d} x) & = z \int_{\mathbb{N}_{-1}} \Big|\log \Big(1+ \frac{x}{z+1} \Big) \Big|^{p} \nu_{\rm br}({\rm d} x) \nonumber \\
& \quad \quad  \quad + \frac{z(z-1)}{2} \int_{\mathbb{N}_{-1}} \Big|\log \Big( 1+\frac{x}{z+1} \Big)\Big|^{p} \nu_{\rm int}({\rm d} x).
\end{align}
\noindent On the one hand, it follows from \eqref{Measuresbrint} and \eqref{logIne1} that, for $z \in \mathbb{N}$, 
\begin{align}  \label{eq12ab}
z \int_{\mathbb{N}_{-1}} \Big| \log \Big(1+ \frac{x}{z+1} \Big) \Big|^{p} \nu_{\rm br}({\rm d} x)  & = zd \Big(\log \Big(1+ \frac{1}{z} \Big) \Big)^{p} + z\sum_{i \geq 1} \pi_{i} \Big(\log \Big(1+ \frac{i}{z+1} \Big) \Big)^{p} \nonumber \\ 
& \leq d + \sum_{i \geq 1} i^{p} \pi_{i}.
\end{align}
\noindent On the other hand, it follows from \eqref{Measuresbrint} and \eqref{logIne1} that, for $z \in \mathbb{N}$,
\begin{align}  \label{eq12ac}
\frac{z(z-1)}{2} \int_{\mathbb{N}_{-1}} \Big|\log \Big(1+ \frac{x}{z+1} \Big) \Big|^{p} \nu_{\rm int}({\rm d} x) & = z(z-1)c \Big(\log \Big(1+ \frac{1}{z} \Big) \Big)^{p} + z(z-1)\sum_{i \geq 1} b_{i} \Big(\log \Big(1+ \frac{i}{z+1} \Big) \Big)^{p} \nonumber \\ 
& \leq   c + \sum_{i \geq 1} i^{p} b_{i}.
\end{align}
\noindent Note that, for $\theta >1$, we have that $x^{p}\leq \frac{p!}{(\ln \theta)^{p}} \theta^{x}$, for $x \in \mathbb{N}$. Then \eqref{Ineq1fact} follows from \eqref{Assump4}, \eqref{eq12}, \eqref{eq12ab}, \eqref{eq12ac} and the above inequality. 

The proof of \eqref{Ineq2fact} is similar, but uses the inequality $\log(1+x) \leq \sqrt{2x}$, for $x \geq 0$, instead of \eqref{logIne1}. 
\end{proof}

For the next result, we assume \eqref{Assump4}, take  $\kappa \geq 0$ as in Lemma \ref{lemma1} \ref{Prop4}, $\theta >1$ and set
\begin{align} 
\eta_{\theta} \coloneqq  \frac{(\ln \theta) \wedge (\ln \theta)^{2}}{24(d\theta + c \theta + \sum_{i \geq 1} \theta^{i} \pi_{i}+ \sum_{i \geq 1} \theta^{i} b_{i})}.
\end{align}
\begin{lemma} \label{lemma3}
Suppose that $\mathbf{m}_{\rm int}<0$ and \eqref{Assump4} hold. Then, for any $\varepsilon >1$ and $t < T(\infty) \wedge \varepsilon/4\kappa \wedge \eta_{\theta}$,
\begin{align}
\mathbb{P}\left( \sup_{u \leq t}  \Big| \log \Big(  \frac{Z_{u}^{(\infty)}+1}{v_{u}+1} \Big) \Big| \geq \varepsilon \right) \leq 3e^{-\frac{\varepsilon \ln \theta}{144}}.
\end{align}
\end{lemma}

\begin{proof} Let $d_{\log}(x,y):= |\log(1+x) - \log(1+y)|$, for $x,y \in [0, \infty)$ and for $\varepsilon >1$, $z \in \mathbb{N}_{0} \cap D$, we introduce
\begin{align}
T_{D, \varepsilon}(z): = \sup \left\{ t \in [0, T(z)): \forall \,  s \leq t, \, \overline{B}_{d_{\log}}(\phi(z, s), \varepsilon) \cap \mathbb{N}_{0} \subset D \right\}.
\end{align}
Suppose that we have proved that for any $\varepsilon >1$, $z \in \mathbb{N}_{0} \cap D$ and $t < T_{D, \varepsilon}(z) \wedge \varepsilon/4\kappa \wedge \eta_{\theta}$, 
\begin{align} \label{eq7}
\mathbb{P}\left( \sup_{u \leq t} \Big| \log \Big(  \frac{Z_{u}^{(z)}+1}{\phi(z, u)+1} \Big) \Big| \geq \varepsilon \right) \leq 3e^{-\frac{\varepsilon \ln \theta}{144}}.
\end{align}
Then, let $T_{0} \in (0,T(\infty) \wedge \varepsilon/4\kappa \wedge \eta_{\theta})$ such that $\overline{B}_{d_{\log}}(\phi(\infty, T_{0}), 2 \varepsilon) \subset D$. Observe that $T_{D, \varepsilon}(z) \geq T_{0}$ for $z$ large enough since $\phi(z, T_{0}) \uparrow v_{T_{0}}$, as $z\uparrow \infty$, and $u \in [0, T(z)) \rightarrow \phi(z, u)$ decreases. Thus, for any $t <T_{0}$,
\begin{align} \label{eq8}
\limsup_{z \rightarrow \infty, z \in \mathbb{N}_{0}} \mathbb{P} \Big( \sup_{u \leq t}  \Big| \log \Big(  \frac{Z_{u}^{(z)}+1}{\phi(z, u)+1} \Big) \Big|\geq \varepsilon \Big) \leq 3e^{-\frac{\varepsilon \ln \theta}{144}}. 
\end{align}
Now, we  also  fix $t_{0} \in (0, T_{0})$. By Lemma \ref{lemma2}, the flow $\phi$ comes down from infinity instantaneously, so $v_{u} = \phi(\infty,u) <\infty$ on $[t_{0}, t]$. By Dini's theorem, $\phi(z, \cdot)$ converges to $v_{\cdot}$ uniformly on $[t_{0}, t]$,  as $z\uparrow \infty$, using the monotonicity of the convergence and the continuity of the limit. We obtain from \eqref{eq8} that for any $t <T_{0}$,
\begin{align}
\limsup_{z \rightarrow \infty, z \in \mathbb{N}_{0}} \mathbb{P} \Big( \sup_{t_{0} \leq u \leq t} \Big| \log \Big(  \frac{Z_{u}^{(z)}+1}{v_{u}+1} \Big) \Big| \geq 2 \varepsilon \Big) \leq 3e^{-\frac{\varepsilon \ln \theta}{144}}
\end{align}
\noindent and the weak convergence of $(Z^{(z)})_{z \in \mathbb{N}_{0}}$ to $Z^{(\infty)}$ (Theorem \ref{Main1}) yields
\begin{align}
\mathbb{P} \Big( \sup_{t_{0} \leq u \leq t} \Big| \log \Big(  \frac{Z_{u}^{(\infty)}+1}{v_{u}+1} \Big) \Big| \geq 2 \varepsilon \Big) \leq 3e^{-\frac{\varepsilon \ln \theta}{144}}. 
\end{align}
\noindent Therefore, our claim follows by letting $t_{0} \downarrow 0$. 

It only remains to prove \eqref{eq7}. Consider the function $F: (-1, \infty) \rightarrow \mathbb{R}$ given by $F(x) = \log(1+x)$. The function $F$ can be extended in such a way that $F \in C^{2}(\mathbb{R}, \mathbb{R})$. Indeed, we may consider $\varphi F$, where $\varphi \in C^{\infty}(\mathbb{R}, \mathbb{R})$ is equal to $0$ on $(-\infty, -1]$ and to $1$ on $\overline{D \cup \mathbb{N}_{0}}$, thanks to the smooth Urysohn lemma since the aforementioned  sets are disjoint and closed. In other words, we may extend $F$ from $\overline{D \cup \mathbb{N}_{0}}$ to $\mathbb{R}$ such that $F \in C^{2}(\mathbb{R}, \mathbb{R})$. Then, by applying It\^o's formula (Theorem 5.1 in Chapter 2 of \cite{Ikeda1981}) and Proposition \ref{Pro1}, for $t \geq 0$, 
\begin{align}
\log(1+Z_{t}^{(z)}) = \log(1+z) + \int_{0}^{t} \int_{\mathcal{X}}  \log\Big( 1+ \frac{H(Z_{s-}^{(z)}, u, y, x)}{Z_{s-}^{(z)}+1} \Big) N({\rm d} s, {\rm d} u, {\rm d} y, {\rm d} x).
\end{align}
Next, let us consider  the functions $h$ and $\psi$ as in Lemma \ref{lemma1}. Then, the process $(\log(1+Z_{t}^{(z)}))_{t \geq 0}$ is a semimartingale that takes values in $ \{\log(1), \log(2), \log(3), \dots\}$ and can be written as
\begin{align}
\log(1+Z_{t}^{(z)}) = \log(1+z)+ \int_{0}^{t} \hat{\psi}(\log(1+Z_{s}^{(z)})) {\rm d} s + A_{t} + M^{\rm d}_{t}, \quad t \geq 0,
\end{align}
\noindent where $\hat{\psi}(y) := \psi(y) \mathbf{1}_{ \{ y\in \log(1+D) \}}$, for $y \in  \mathbb{R}$ and $\log(1+D) :=\{ \log(1+x): x \in D\}$, is a Borel locally bounded function by Lemma \ref{lemma1}, 
\begin{align} \label{eq43}
A_{t} := \int_{0}^{t} h(Z_{s}^{(z)})\mathbf{1}_{\{\log(1+Z_{s}^{(z)}) \not \in \log(1+D) \}} {\rm d} s 
\end{align}
\noindent is a continuous adapted process with a.s.\ bounded variations paths and
\begin{align} \label{eq44}
M^{\rm d}_{t} := \int_{0}^{t} \int_{\mathcal{X}}  \log\Big( 1+ \frac{H(Z_{s-}^{(z)}, u, y, x)}{Z_{s-}^{(z)}+1} \Big)  \tilde{N}({\rm d} s, {\rm d} u, {\rm d} y, {\rm d} x)
\end{align}
\noindent is a c\`adl\`ag local martingale. Indeed, by \eqref{Ineq1fact} in Lemma \ref{lemma7}, $(M^{\rm d}_{t})_{t \geq 0}$ is a locally square-integrable martingale. 

Note that the dynamical system $y_{t} = \log(1+\phi(z,t))$ satisfies for $t < T(z)$,
\begin{align}
y_{0} = \log(1+z), \quad y_{t}^{\prime} = \hat{\psi}(y_{t}) = \psi(y_{t}),
\end{align}
\noindent since $\psi = \hat{\psi}$ on $\log(1+D)$. This flow is thus associated with the vector field $\hat{\psi}$ which is locally Lipschitz on $\log(1+D)$ by Lemma \ref{lemma1}. Set $T^{\prime}(y_{0}) = T(z)$, and for $\varepsilon>0$, let 
\begin{align}
T_{\log(1+D), \varepsilon}(y_{0}) \coloneqq \sup\left\{t \in [0, T^{\prime}(y_{0})): \, \forall s \leq t, \, \bar{B}(y_{s},\varepsilon) \cap \log(1+\mathbb{N}_{0}) \subset \log(1+D)  \right\}
\end{align}
\noindent be the last time when $y_{t}$ starting from $y_{0}$ is at distance $\varepsilon >0$ from the boundary of $\log(1+D)$ under the Euclidean distance. Note that $T_{\log(1+D), \varepsilon}(y_{0})  \geq T_{D, \varepsilon}(z)$ and that $\hat{\psi}$ is $(0, \alpha)$ nonexpansive on $\log(1+D)$ by Lemma \ref{lemma1}. By applying Lemma 2.1 in \cite{Bansaye2019} to $(\log(1+Z_{t}^{(z)}))_{t \geq 0}$ with $\log(1+Z_{0}^{(z)}) = y_{0}=\log(1+z)$, we obtain that, for any $t< T_{D, \varepsilon}(z) \wedge \varepsilon/4\kappa \wedge \eta_{\theta}$, 
\begin{align} \label{eq42}
\mathbb{P} \left( S_{t} \geq \varepsilon  \right) \leq \mathbb{P} \left( \sup_{u \leq t} \tilde{R}_{u}^{\varepsilon} \geq \frac{\varepsilon^{2}}{2}  \right),
\end{align}
\noindent where for $u \leq t$,
\begin{align} \label{eq41}
S_{u} := \sup_{r \leq u} d_{\log}(Z_{r}^{(z)},\phi(z, r)) = \sup_{r \leq u} |\log(1+Z_{r}^{(z)}))-y_{r}|
\end{align}
\noindent and 
\begin{align} 
\tilde{R}_{u}^{\varepsilon} & := \mathbf{1}_{\{S_{u-} \leq \varepsilon \}} \Big( 2\int_{0}^{u}  (\log(1+Z_{s-}^{(z)})-y_{s}) {\rm d} (A_{s} + M^{\rm d}_{s}) + \sum_{s \leq u} |\Delta \log(1+Z_{s}^{(z)})|^{2} \Big).
\end{align}
Then, by \eqref{eq42} and the union bound, for any $t< T_{D, \varepsilon}(z) \wedge \varepsilon/4\kappa \wedge \eta_{\theta}$, 
\begin{align} \label{eq9}
\mathbb{P} \left( \sup_{u \leq t} \Big| \log \Big(  \frac{Z_{u}^{(z)}+1}{\phi(z, u)+1} \Big) \Big| \geq \varepsilon \right)  & \leq \mathbb{P} \left( \sup_{u \leq t}  \mathbf{1}_{\{S_{u-} \leq \varepsilon \}} \int_{0}^{u} (\log(1+Z_{s-}^{(z)})-y_{s}) {\rm d} A_{s} \geq \frac{\varepsilon^{2}}{12} \right)  \nonumber \\
& \quad \quad  + \mathbb{P} \left( \sup_{u \leq t}  \mathbf{1}_{\{S_{u-} \leq \varepsilon \}} \int_{0}^{u}  (\log(1+Z_{s-}^{(z)})-y_{s}) {\rm d} M^{\rm d}_{s} \geq \frac{\varepsilon^{2}}{12} \right)  \nonumber \\
& \quad \quad  + \mathbb{P} \left( \sum_{u \leq t} \mathbf{1}_{\{S_{u-} \leq \varepsilon \}} |\Delta \log(1+Z_{u}^{(z)})|^{2} \geq \frac{\varepsilon^{2}}{6} \right). 
\end{align} 

\noindent On the event $\{S_{u-} \leq \varepsilon \}$, we have that $\log(1+Z_{u-}^{(z)}) \in \log(1+D)$, for any $u \leq t <  T_{D, \varepsilon}(z) \wedge \varepsilon/4\kappa \wedge \eta_{\theta}$. So, by \eqref{eq43},
\begin{align} \label{eq5}
\mathbb{P} \left( \sup_{u \leq t} \mathbf{1}_{\{S_{u-} \leq \varepsilon \}} \int_{0}^{u}  (\log(1+Z_{s-}^{(z)})-y_{s}) {\rm d} A_{s} \geq \frac{\varepsilon^{2}}{12} \right) =0. 
\end{align}
\noindent Let
\begin{align}
J_{r}^{(\varepsilon)} & \coloneqq \int_{0}^{r} \mathbf{1}_{\{S_{s-} \leq \varepsilon \}} (\log(1+Z_{s-}^{(z)}))-y_{s}) {\rm d} M^{\rm d}_{s}, \quad \text{for}  \, \,  r \leq t. 
\end{align}
\noindent Note that $(J_{r}^{(\varepsilon)})_{r \in [0,t]}$ is a local martingale. Then, 
\begin{align} \label{eq49}
& \mathbb{P} \left( \sup_{u \leq t}  \mathbf{1}_{\{S_{u-} \leq \varepsilon \}} \int_{0}^{u}  (\log(1+Z_{s-}^{(z)})-y_{s}) {\rm d} M^{\rm d}_{s} \geq \frac{\varepsilon^{2}}{12} \right) \nonumber \\
& \qquad \qquad =  \mathbb{P} \left( \sup_{u \leq t}  \mathbf{1}_{\{S_{u-} \leq \varepsilon \}} J_{u}^{(\varepsilon)} \geq \frac{\varepsilon^{2}}{12}, S_{t} \leq \varepsilon \right) + \mathbb{P} \left( \sup_{u \leq t}  \mathbf{1}_{\{S_{u-} \leq \varepsilon \}} J_{u}^{(\varepsilon)} \geq \frac{\varepsilon^{2}}{12}, S_{t} > \varepsilon \right)
& \nonumber \\
& \qquad \qquad \leq  \mathbb{P} \left( \sup_{u \leq t}  J_{u}^{(\varepsilon)} \geq \frac{\varepsilon^{2}}{12}\right) + \mathbb{P} \left( \sup_{u \leq t}  \mathbf{1}_{\{S_{u-} \leq \varepsilon \}} J_{u}^{(\varepsilon)} \geq \frac{\varepsilon^{2}}{12}, S_{t} > \varepsilon \right).
\end{align} 
\noindent Then, by \eqref{eq9}, \eqref{eq5} and \eqref{eq49},
\begin{align} \label{eq50} 
\mathbb{P} \left( \sup_{u \leq t} \Big| \log \Big(  \frac{Z_{u}^{(z)}+1}{\phi(z, u)+1} \Big) \Big| \geq \varepsilon \right)  & \leq \mathbb{P} \left( \sup_{u \leq t}  J_{u}^{(\varepsilon)} \geq \frac{\varepsilon^{2}}{12}\right) + \mathbb{P} \left( \sup_{u \leq t}  \mathbf{1}_{\{S_{u-} \leq \varepsilon \}} J_{u}^{(\varepsilon)} \geq \frac{\varepsilon^{2}}{12}, S_{t} > \varepsilon \right) \nonumber \\
& \quad \quad  + \mathbb{P} \left( \sum_{u \leq t} \mathbf{1}_{\{S_{u-} \leq \varepsilon \}} |\Delta \log(1+Z_{u}^{(z)})|^{2} \geq \frac{\varepsilon^{2}}{6} \right). 
\end{align} 

Next, we bound each of the probabilities on the right-hand side of \eqref{eq50}.  We start with the first term. By the Burkholder-Davis-Gundy inequality and \eqref{eq41}, there exists a constant $C>0$ such that, for $r \leq t$,
\begin{align}
\mathbb{E}[(J_{r}^{(\varepsilon)})^{2}] \leq C \mathbb{E} \left[ \int_{0}^{r} \mathbf{1}_{\{S_{s-} \leq \varepsilon \}} (\log(1+Z_{s-}^{(z)}))-y_{s})^{2} {\rm d} [M^{\rm d}]_{s} \right]  \leq C \varepsilon^{2} \mathbb{E} \left[ [M^{\rm d}]_{r} \right], 
\end{align}
\noindent where $[M^{\rm d}]$ denotes the quadratic variation of $(M^{\rm d}_{s})_{s \geq 0}$. In particular, by \eqref{Ineq1fact} in Lemma \ref{lemma7},
\begin{align}
\mathbb{E} \left[ [M^{\rm d}]_{r} \right] = \mathbb{E} \left[\int_{0}^{r} \int_{\mathcal{X}} \Big(\log \Big(1+\frac{H(Z_{s-}^{(z)}, u, y, x)}{Z_{s-}^{(z)}+1} \Big) \Big)^{2} q({\rm d} u, {\rm d} y, {\rm d} x) {\rm d} s \right] < \infty.
\end{align}
\noindent Thus, $(J_{r}^{(\varepsilon)})_{r \in [0,t]}$ is a locally square-integrable martingale. On the other hand, for $p \geq 2$, we note by \eqref{Ineq1fact} in Lemma \ref{lemma7} that the (predictable) compensator $(V^{(p)}_{r})_{r \in [0,t]}$ of the process $(\sum_{s \leq r} |\Delta J_{s}^{(\varepsilon)}|^{p})_{r \in [0,t]}$ satisfies
\begin{align} \label{eq51}
V^{(p)}_{r} \leq \frac{p!}{2} \Big(\frac{\varepsilon}{\ln \theta} \Big)^{p} C_{\theta}r, \quad \text{for}  \, \,  r \leq t,
\end{align}
\noindent where 
\begin{equation}\label{ctheta}
C_{\theta} \coloneqq  2 \left(d\theta + c \theta + \sum_{i \geq 1} \theta^{i} \pi_{i}+\sum_{i \geq 1} \theta^{i} b_{i}\right).
\end{equation}
Then Lemma \ref{lemmaA1} implies that, for any $t< T_{D, \varepsilon}(z) \wedge \varepsilon/4\kappa \wedge \eta_{\theta}$, 
\begin{align} \label{eq10}
\mathbb{P}\left( \sup_{u \leq t}  J_{u}^{(\varepsilon)} \geq \frac{\varepsilon^{2}}{12} \right) & \leq \exp \left( - \frac{\varepsilon^{4}}{2(12)^{2} \left( \frac{\varepsilon^{3}}{12 \ln \theta} + \frac{\varepsilon^{2}C_{\theta}t}{(\ln \theta)^{2}} \right)   } \right) \leq \exp\left(-\frac{\varepsilon \ln \theta}{48}\right),
\end{align}
\noindent where we have used that $\varepsilon >1$ and $t < \eta_{\theta} = \frac{(\ln \theta) \wedge (\ln \theta)^{2}}{12 C_{\theta}}$. 

Next, we bound the second term on the right-hand side of \eqref{eq50}. Let $\sigma \coloneqq \inf \{ 0 < s \leq t: J_{s}^{(\varepsilon)} \geq \varepsilon^{2}/12 \}$. Note that $\sigma$ is a stopping time and that $\sigma \leq t$. For $0 < \lambda < (\ln \theta) / \varepsilon$, note that, by \eqref{eq51},
\begin{align} \label{eqA2}
L_{r}^{(\varepsilon)}  \coloneqq \sum_{p=2}^{\infty} \frac{\lambda^{p}}{p!} V^{(p)}_{r} \leq \frac{\lambda^{2} (\frac{\varepsilon}{\ln \theta})^{2} C_{\theta}t}{2(1-\frac{\lambda \varepsilon}{\ln \theta})}, \quad \text{for} \quad r \leq t. 
\end{align} 
\noindent By Lemma \ref{lemmaA1}, $(\exp(\lambda J_{r}^{(\varepsilon)}-L_{r}^{(\varepsilon)} ))_{r \in [0,t]}$ is a super-martingale. Consider the event 
\begin{align} \label{eq52} 
A \coloneqq \Big\{ \sup_{u \leq t}  \mathbf{1}_{\{S_{u-} \leq \varepsilon \}} J_{u}^{(\varepsilon)} \geq \frac{\varepsilon^{2}}{12}, S_{t} > \varepsilon \Big\}.
\end{align}
\noindent In particular, we have that
\begin{align}  \label{eq53}
\mathbb{E} \left[ \mathbf{1}_{A} \exp(\lambda J_{\sigma}^{(\varepsilon)}-L_{\sigma}^{(\varepsilon)}) \right] \leq 1.
\end{align}
\noindent Furthermore, since
\begin{align} \label{eq54}
\exp(\lambda J_{\sigma}^{(\varepsilon)}-L_{\sigma}^{(\varepsilon)}) \geq \exp \Big( \frac{\lambda \varepsilon^{2}}{12}  - \frac{\lambda^{2} (\frac{\varepsilon}{\ln \theta})^{2} C_{\theta}t}{2(1-\frac{\lambda \varepsilon}{\ln \theta})} \Big),
\end{align}
\noindent we conclude from \eqref{eq52}, \eqref{eq53} and \eqref{eq54} that
\begin{align}
\mathbb{E}[\mathbf{1}_{A}] \exp \Big( \frac{\lambda \varepsilon^{2}}{12}  - \frac{\lambda^{2} (\frac{\varepsilon}{\ln \theta})^{2} C_{\theta}t}{2(1-\frac{\lambda \varepsilon}{\ln \theta})} \Big) \leq 1.
\end{align}
\noindent Thus, by choosing $\lambda= (\frac{\varepsilon^{2}}{12})/ ((\frac{\varepsilon}{\ln \theta})^{2} C_{\theta}t + \frac{\varepsilon^{3}}{12 \ln \theta})$, we deduce that
\begin{align} \label{eq55}
\mathbb{P} \left( \sup_{u \leq t}  \mathbf{1}_{\{S_{u-} \leq \varepsilon \}} J_{u}^{(\varepsilon)} \geq \frac{\varepsilon^{2}}{12}, S_{t} > \varepsilon \right) \leq \exp \left( - \frac{\varepsilon^{4}}{2(12)^{2} \left( \frac{\varepsilon^{3}}{12 \ln \theta} + \frac{\varepsilon^{2}C_{\theta}t}{(\ln \theta)^{2}} \right)   } \right) \leq \exp\left(-\frac{\varepsilon \ln \theta}{48}\right).
\end{align}

Finally, we bound the remaining term in \eqref{eq50}. Note that 
\begin{align}
\sum_{s \leq r} \mathbf{1}_{\{S_{s-} \leq \varepsilon \}} |\Delta \log(1+Z_{s}^{(z)}))|^{2} = \int_{0}^{r} \int_{\mathcal{X}} \mathbf{1}_{\{S_{s-} \leq \varepsilon \}}   \Big(\log\Big( 1+ \frac{H(Z_{s-}^{(z)}, u, y, x)}{Z_{s-}^{(z)}+1} \Big) \Big)^{2} N({\rm d} s, {\rm d} u, {\rm d} y, {\rm d} x),
\end{align}
\noindent for $r \leq t$. In particular, by \eqref{Ineq1fact} in Lemma \ref{lemma7}, the process $(\hat{J}_{r}^{(\varepsilon)})_{r \in [0, t]}$, given by
\begin{align}
\hat{J}_{r}^{(\varepsilon)} := \int_{0}^{r}   \int_{\mathcal{X}} \mathbf{1}_{\{S_{s-} \leq \varepsilon \}}   \Big(\log\Big( 1+ \frac{H(Z_{s-}^{(z)}, u, y, x)}{Z_{s-}^{(z)}+1} \Big) \Big)^{2}  \tilde{N}({\rm d} s, {\rm d} u, {\rm d} y, {\rm d} x),
\end{align}
\noindent is a locally square-integrable martingale. Again, by \eqref{Ineq1fact} in Lemma \ref{lemma7}, we see 
\begin{align} \label{eq6}
\int_{0}^{r}  \int_{\mathcal{X}} \mathbf{1}_{\{S_{s-} \leq \varepsilon \}} \Big(\log\Big( 1+ \frac{H(Z_{s-}^{(z)}, u, y, x)}{Z_{s-}^{(z)}+1} \Big) \Big)^{2}  q({\rm d} u, {\rm d} y, {\rm d} x)  {\rm d} s \leq \frac{C_{\theta}}{(\ln \theta)^{2}} r, 
\end{align}
\noindent for $r \leq t$. For $p \geq 2$, the inequality \eqref{Ineq2fact} in Lemma \ref{lemma7} implies that the (predictable) compensator $(\hat{V}^{(p)}_{r})_{r \in [0, t]}$ of $(\sum_{s \leq r} |\Delta \hat{J}_{s}^{(\varepsilon)}|^{p})_{r \in [0, t]}$ satisfies $\hat{V}^{(p)}_{r} \leq \frac{p!}{2}(\frac{2}{\ln \theta})^{p}C_{\theta} r$, for $r \leq t$, where $C_{\theta}$ is as in \eqref{ctheta}. Then, it follows from \eqref{eq6} and Lemma \ref{lemmaA1} that, for $t < T_{D, \varepsilon}(z) \wedge \varepsilon/4\kappa \wedge \eta_{\theta}$,
\begin{align} \label{eq11}
\mathbb{P}\left( \sum_{u \leq t} \mathbf{1}_{\{S_{u-} \leq \varepsilon \}} |\Delta \log(1+Z_{u}^{(z)})|^{2} \geq \frac{\varepsilon^{2}}{6} \right) & \leq \mathbb{P}\Big( \hat{J}_{t}^{(\varepsilon)}  \geq \frac{\varepsilon^{2}}{12} \Big) + \mathbf{1}_{\left \{ \frac{C_{\theta}}{(\ln \theta)^{2}}t \geq \frac{\varepsilon^{2}}{12} \right\}} \nonumber \\
& \leq  \exp \left( - \frac{\varepsilon^{4}}{2(12)^{2} \left( \frac{\varepsilon^{2}}{6 \ln \theta} + \frac{4C_{\theta}t}{(\ln \theta)^{2}} \right)   } \right) \nonumber \\
& \leq \exp\left( -\frac{\varepsilon \ln \theta}{144}\right),
\end{align}
\noindent where we have used that $\varepsilon >1$, $\ln \theta >0$ (since $\theta >1$) and $t < \eta_{\theta} = \frac{(\ln \theta) \wedge (\ln \theta)^{2}}{12 C_{\theta}}$. 

Finally, \eqref{eq7} follows by combining \eqref{eq50}, \eqref{eq10}, \eqref{eq55} and \eqref{eq11}. 
\end{proof}

We continue with the proof of Theorem \ref{PropC2}.

\begin{proof}[Proof of Theorem \ref{PropC2}] 
Suppose that we have proved that, for $p \in (0,2]$, 
\begin{align} \label{PropC1}
\lim_{t \downarrow 0}\mathbb{E}\left[ \sup_{s \leq t}  \left| \frac{Z_{s}^{(\infty)}+1}{v_{s}+1} - 1 \right|^{p} \right]= 0. 
\end{align}
\noindent For $s \geq 0$, note that
\begin{align}
-\mathbf{m}_{\rm int} sZ_{s}^{(\infty)}-1 = -\mathbf{m}_{\rm int} s(v_{s}+1) \Big(\frac{Z_{s}^{(\infty)}+1}{v_{s}+1} -1 \Big) - \mathbf{m}_{\rm int} s(v_{s}+1) + s \mathbf{m}_{\rm int}-1
\end{align}
\noindent and thus, \eqref{PropC1} and Lemma \ref{lemma6} imply our claim in Theorem \ref{PropC2}.

Now we prove \eqref{PropC1}. Note that
\begin{align}
\sup_{s \leq t} \left| \log\left(\frac{Z_{s}^{(\infty)}+1}{v_{s}+1}  \right) \right| =  \log \left(  \sup_{s \leq t}  \left| \frac{Z_{s}^{(\infty)}+1}{v_{s}+1}\right| \vee \sup_{s \leq t}  \left| \frac{v_{s}+1}{Z_{s}^{(\infty)}+1}\right| \right), \quad t \geq 0.
\end{align}
\noindent Fix $0 \leq t < T(\infty)$, and choose $\varepsilon >e$ (i.e., $\ln \varepsilon >1$) such that $t < T(\infty) \wedge (\ln \varepsilon)/4\kappa \wedge \eta_{\theta}$. It follows from Lemma \ref{lemma3} that, for $\theta >1$ and $u > \varepsilon$, 
\begin{align}
\mathbb{P}\left( \sup_{s \leq t}  \left| \frac{Z_{s}^{(\infty)}+1}{v_{s}+1}\right| \geq u \right) \leq 3e^{-\frac{(\ln u) \ln \theta}{144}} = 3u^{-\frac{\log \theta}{144}};
\end{align}
\noindent to see this note that $(\ln \varepsilon)/4\kappa \leq (\ln u)/4\kappa$, and thus $t < T(\infty) \wedge (\ln u)/4\kappa \wedge \eta_{\theta}$. By choosing $\theta = \vartheta > e^{288}$ as in \eqref{Assump4}, we deduce that, for $p \in [1, 2]$, there is a constant $0 < C_{1}(p) < \infty$ such that 
\begin{align} \label{eq13}
\mathbb{E} \left[ \sup_{s \leq t}  \left| \frac{Z_{s}^{(\infty)}+1}{v_{s}+1} \right|^{p} \right] < C_{1}(p). 
\end{align}
In particular, for a different constant $0 < C_{2}(p) < \infty$,
\begin{align} 
\mathbb{E}\left[ \sup_{s \leq t}  \left| \frac{Z_{s}^{(\infty)}+1}{v_{s}+1} - 1 \right|^{p} \right] < C_{2}(p).
\end{align} 
\noindent Therefore, for $p \in [1,2]$, \eqref{PropC1} follows by \eqref{Pro2} and the dominated convergence theorem. The case  $p \in (0,1)$  follows by Jensen's inequality. This completes the proof.
\end{proof}

\begin{corollary} \label{corollaryII}
Suppose that $\mathbf{m}_{\rm int}<0$ and \eqref{Assump4} hold. We have the following:
\begin{enumerate}[label=(\roman*)]
\item \label{PropC3} For $p \in(0,2]$ and any $0 \leq  T < T(\infty)$, there exists a constant $0 < C(p) < \infty$ such that 
\begin{align}
\mathbb{E}\Big[ \sup_{s \leq t}  \Big| - \mathbf{m}_{\rm int} s Z_{s}^{(\infty)} \Big|^{p} \Big] < C(p), \quad \text{for} \quad t \leq T. 
\end{align} 

\item \label{PropC4} For all $0 <  r \leq t < T(\infty)$, 
\begin{align}
\mathbb{E}\Big[ \sup_{r \leq s \leq t} Z_{s}^{(\infty)} \Big] < \infty \qquad \text{and} \qquad \mathbb{E}\Big[ \sup_{r \leq s \leq t} (Z_{s}^{(\infty)})^{2} \Big] < \infty.
\end{align}
\end{enumerate}
\end{corollary}

\begin{proof}
Note that 
\begin{align} \label{eq14}
- \mathbf{m}_{\rm int} sZ_{s}^{(\infty)} \leq - \mathbf{m}_{\rm int} s(v_{s}+1) \frac{Z_{s}^{(\infty)}+1}{v_{s}+1}.
\end{align}
\noindent Recall that $- \mathbf{m}_{\rm int} >0$. By Lemma \ref{lemma6}, $\lim_{s \downarrow 0} - \mathbf{m}_{\rm int} s(v_{s}+1) =1$. In particular, $s \mapsto - \mathbf{m}_{\rm int} s(v_{s}+1)$ is bounded on $[0, T]$. Then, \eqref{eq14}, Lemma \ref{lemma6} and \eqref{eq13} imply \ref{PropC3} for $p \in[1,2]$. For $p \in (0,1)$, the inequality in \ref{PropC3} follows by Jensen's inequality. To prove \ref{PropC4}, note that, for $0 < r \leq s \leq t$, 
\begin{align}
Z_{s}^{(\infty)} \leq (-r \mathbf{m}_{\rm int})^{-1} \sup_{r\leq u \leq t} \left|- \mathbf{m}_{\rm int} uZ_{u}^{(\infty)}\right|.
\end{align}
Then, our claim follows from part \ref{PropC3}.
\end{proof}

\subsection{Proof of Theorem \ref{Main3}} \label{ProofofMain3}

In this section, we prove Theorem \ref{Main3}. We follow the approach from \cite{Vlada20152} (see also \cite{Vlada2015}), modifying it to account for the branching and cooperative mechanisms of the BPI process. Recall the definition of $T(\infty)$ given before \eqref{Def1Speed}.
\begin{lemma} \label{lemma8}
Suppose that $\mathbf{m}_{\rm int}<0$ and \eqref{Assump4} hold. For all $0 < r \leq t < T(\infty)$, we have that
\begin{align}
Z_{t}^{(\infty)} = Z_{r}^{(\infty)} + \int_{(r, t]} \int_{\mathcal{X}} H(Z_{s-}^{(\infty)}, u, y, x) N({\rm d} s, {\rm d} u, {\rm d} y, {\rm d} x).
\end{align}
\end{lemma}

\begin{proof}
Recall that \eqref{Assump4} implies $\mathbf{m}_{\rm br}<\infty$, the latter together with $\mathbf{m}_{\rm int}<0$ implies that the BPI process does not explode in finite time; see Example 2.1 in \cite{Berzunza2020}. By Proposition \ref{Pro1}, for all $z \in \mathbb{N}_{0}$ and $0 < r \leq t < T(\infty)$,
\begin{align} \label{Integral1}
Z_{t}^{(z)} & = Z_{r}^{(z)} + \int_{(r, t]} \int_{\mathcal{X}} H(Z_{s-}^{(z)}, u, y, x) N({\rm d} s, {\rm d} u, {\rm d} y, {\rm d} x) \nonumber \\
& = Z_{r}^{(z)} + \int_{(r, t]} \int_{\mathcal{X}} H(Z_{s-}^{(z)}, u, y, -1) \mathbf{1}_{\{ x=-1 \}} N({\rm d} s, {\rm d} u, {\rm d} y, {\rm d} x) \nonumber \\
& \quad \quad \quad + \int_{(r, t]} \int_{\mathcal{X}} H(Z_{s-}^{(z)}, u, y, x) \mathbf{1}_{\{ x \geq 0 \}} N({\rm d} s, {\rm d} u, {\rm d} y, {\rm d} x). 
\end{align}
\noindent Proposition \ref{Prop6} and Theorem \ref{Main1} imply that, for any $0  < t < T(\infty)$, $\lim_{z \rightarrow \infty} Z_{t}^{(z)} =  Z_{t}^{(\infty)}$ (increasingly), almost surely. Recall that the Poisson random measure $N$ is an almost-surely $\sigma$-finite measure. Since $z \in \mathbb{R}_{+} \mapsto H(z, u, y, x) \mathbf{1}_{\{ x \geq 0 \}}$ is an increasing function, the monotone convergence theorem implies that
\begin{align}
& \lim_{z \rightarrow \infty} \int_{(r, t]} \int_{\mathcal{X}} H(Z_{s-}^{(z)}, u, y, x) \mathbf{1}_{\{ x \geq 0 \}} N({\rm d} s, {\rm d} u, {\rm d} y, {\rm d} x) \nonumber \\
& \quad \quad \quad = \int_{(r, t]} \int_{\mathcal{X}} H(Z_{s-}^{(\infty)}, u, y, x) \mathbf{1}_{\{ x \geq 0 \}} N({\rm d} s, {\rm d} u, {\rm d} y, {\rm d} x).
\end{align}
\noindent On the other hand, $|H(Z_{s-}^{(z)}, u, y, -1)| \mathbf{1}_{\{ x=-1 \}} \leq |H(Z_{s-}^{(\infty)}, u, y, -1)| \mathbf{1}_{\{ x=-1 \}}$, for all $r \leq s \leq t < T(\infty)$ and $(u, y, x) \in \mathcal{X}$. By Corollary \ref{corollaryII} \ref{PropC4}, one can easily deduce that $|H(Z_{s-}^{(\infty)}, u, y, -1)| \mathbf{1}_{\{ x=-1 \}}$ is integrable with respect to the Poisson random measure $N$, almost surely. Therefore, by the dominated convergence theorem,
\begin{align}
& \lim_{z \rightarrow \infty} \int_{(r, t]} \int_{\mathcal{X}} H(Z_{s-}^{(z)}, u, y, -1) \mathbf{1}_{\{ x = -1 \}} N({\rm d} s, {\rm d} u, {\rm d} y, {\rm d} x) \nonumber \\
& \quad \quad \quad = \int_{(r, t]} \int_{\mathcal{X}} H(Z_{s-}^{(\infty)}, u, y, -1) \mathbf{1}_{\{ x = -1 \}} N({\rm d} s, {\rm d} u, {\rm d} y, {\rm d} x).
\end{align}
\noindent Clearly, the combination of the above limits and \eqref{Integral1} imply our claim. 
\end{proof}

\begin{lemma} \label{lemma9}
Suppose that $\mathbf{m}_{\rm int}<0$ and \eqref{Assump4} hold. Then, for any $0 < r \leq t < T(\infty)$, 
\begin{align} \label{eq15}
\mathbb{E}\Big[ \Big| \int_{r}^{t} (\mathbf{m}_{\rm br} Z_{s}^{(\infty)} + \mathbf{m}_{\rm int}  Z_{s}^{(\infty)} (Z_{s}^{(\infty)}-1)) {\rm d} s  \Big| \Big] < \infty.
\end{align}
\noindent In particular, 
\begin{align} \label{eq16}
Z_{t}^{(\infty)} & = Z_{r}^{(\infty)} + \int_{r}^{t} (\mathbf{m}_{\rm br} Z_{s}^{(\infty)} + \mathbf{m}_{\rm int} Z_{s}^{(\infty)} (Z_{s}^{(\infty)}-1)) {\rm d} s \nonumber \\
& \quad \quad \quad + \int_{(r,t]} \int_{\mathcal{X}} H(Z_{s-}^{(\infty)}, u, y, x) \tilde{N}({\rm d} s, {\rm d} u, {\rm d} y, {\rm d} x).
\end{align}
\noindent Moreover, 
\begin{align} \label{eq17}
- \mathbf{m}_{\rm int} t Z_{t}^{(\infty)} & = - \mathbf{m}_{\rm int} r Z_{r}^{(\infty)}  - \int_{r}^{t} (\mathbf{m}_{\rm int} \mathbf{m}_{\rm br}sZ_{s}^{(\infty)} + \mathbf{m}_{\rm int} Z_{s}^{(\infty)} + (\mathbf{m}_{\rm int})^{2}s Z_{s}^{(\infty)} (Z_{s}^{(\infty)}-1) ) {\rm d} s \nonumber \\
& \quad \quad \quad - \mathbf{m}_{\rm int} \int_{(r,t]} \int_{\mathcal{X}}  s H(Z_{s-}^{(\infty)}, u, y, x) \tilde{N}({\rm d} s, {\rm d} u, {\rm d} y, {\rm d} x).
\end{align}
\end{lemma}

\begin{proof}
Corollary \ref{corollaryII} \ref{PropC4} implies \eqref{eq15} and thus, Lemma \ref{lemma8} implies \eqref{eq16} (see e.g., Theorem 8.23 in  \cite{Peszat2007}). Integration by parts and \eqref{eq16} imply \eqref{eq17}. 
\end{proof}

Recall, from the beginning of Section \ref{IntegralR}, the definitions of the Poisson random measures $N_{\rm br}$ and $N_{\rm int}$ on $\mathbb{R}_{+} \times \mathbb{N} \times \mathbb{N}_{-1}$ and $\mathbb{R}_{+} \times \Delta \times \mathbb{N}_{-1}$ with intensity measures ${\rm d} s \otimes q_{\rm br}({\rm d} y, {\rm d} x)$ and ${\rm d} s  \otimes q_{\rm int}({\rm d} y, {\rm d} x)$, respectively, where 
\begin{align*} 
q_{\rm br}({\rm d} y, {\rm d} x) \coloneqq \sum_{i \geq 1}\delta_{i}({\rm d} y)   \otimes \nu_{\rm br}({\rm d} x) \qquad \textrm{and}\qquad
q_{\rm int}({\rm d} y, {\rm d} x) \coloneqq \sum_{(i,j) \in \Delta}\delta_{(i,j)}({\rm d} y) \otimes \nu_{\rm int}({\rm d} x),
\end{align*}
as well as the notation $\Delta_{k} = \{ (i,j) \in \Delta: 1 \leq i <j\leq k\}$. We also let $\tilde{N}_{\rm br}$ and $\tilde{N}_{\rm int}$ be the compensated Poisson random measures of $N_{\rm br}$ and $N_{\rm int}$, respectively. 

\begin{lemma}  \label{lemma10}
Suppose that $\mathbf{m}_{\rm int}<0$ and \eqref{Assump4} hold. Fix $0 < t^{\ast} < T(\infty)$. The processes $(M_{t}^{\rm br})_{t \geq 0}$ and $(M_{t}^{\rm int})_{t \geq 0}$ given by
\begin{align} \label{eq31}
M_{t}^{\rm br} := - \mathbf{m}_{\rm int}  \int_{0}^{t \wedge t^{\ast}} \int_{\mathbb{N}} \int_{\mathbb{N}_{-1}}  s x \mathbf{1}_{\{ 0 < y \leq Z_{s-}^{(\infty)}\}} \tilde{N}_{\rm br}({\rm d} s, {\rm d} y, {\rm d} x)
\end{align}
\noindent and 
\begin{align} \label{eq31b}
M_{t}^{\rm int} := - \mathbf{m}_{\rm int}  \int_{0}^{t \wedge t^{\ast}} \int_{\Delta} \int_{\mathbb{N}_{-1}}  s x \mathbf{1}_{\{ y \in \Delta_{Z_{s-}^{(\infty)}} \}} \tilde{N}_{\rm int}({\rm d} s, {\rm d} y, {\rm d} x), \quad \text{for} \quad t \geq 0,
\end{align}
\noindent are well-defined, locally square-integrable martingales with predictable quadratic variation $(\langle M^{\rm br} \rangle_{t})_{t \geq 0}$ and $(\langle M^{\rm int} \rangle_{t})_{t \geq 0}$ given by
\begin{align} \label{eq18}
\langle M^{\rm br}\rangle_{t} = (\mathbf{m}_{\rm int})^{2} \boldsymbol{\sigma}^{2}_{{\rm br}} \int_{0}^{t \wedge t^{\ast}} s^{2}Z_{s}^{(\infty)} {\rm d}s 
\end{align}
\noindent and 
\begin{align} \label{eq18b}
\langle M^{\rm int}\rangle_{t} = (\mathbf{m}_{\rm int})^{2} \boldsymbol{\sigma}^{2}_{{\rm int}} \int_{0}^{t \wedge t^{\ast}} s^{2}Z_{s}^{(\infty)}(Z_{s}^{(\infty)}-1) {\rm d}s, \quad \text{for} \quad t \geq 0.
\end{align}
\noindent Moreover, for $p \in (0,2]$ and any $0 \leq T <\infty$, there exist constants $0 < C_{1}(p), C_{2}(p) <\infty$ such that 
\begin{align} \label{eq19}
\mathbb{E}\Big[ \sup_{s \leq t} |M_{s}^{\rm br}|^{p} \Big] \leq  C_{1}(p) t^{p}
\end{align}
\noindent and 
\begin{align} \label{eq19b}
\mathbb{E}\Big[ \sup_{s \leq t} |M_{s}^{\rm int}|^{p} \Big] \leq  C_{2}(p) t^{p/2}, \quad \text{for} \quad t \leq T.
\end{align} 
\end{lemma}

\begin{proof}
On the one hand, $x \mathbf{1}_{\{ 0 < y \leq z\}} = 0$, for $z=0$, $y \in \mathbb{N}$ and $x \in \mathbb{N}_{-1}$. On the other hand, $x \mathbf{1}_{\{ y \in \Delta_{z}\}}=0$, for $z\in\{0,1\}$, $y \in \Delta$ and $x \in \mathbb{N}_{-1}$. Then, 
\begin{align}
\int_{0}^{t \wedge t^{\ast}} \int_{\mathbb{N}} \int_{\mathbb{N}_{-1}}  s^{2} x^{2} \mathbf{1}_{\{ 0 < y \leq Z_{s-}^{(\infty)}\}} q_{\rm br}({\rm d} y, {\rm d} x) {\rm d} s & = \boldsymbol{\sigma}^{2}_{{\rm br}} \int_{0}^{t \wedge t^{\ast}} s^{2}Z_{s}^{(\infty)} {\rm d}s 
\end{align}
\noindent and 
\begin{align}
\int_{0}^{t \wedge t^{\ast}} \int_{\Delta} \int_{\mathbb{N}_{-1}}  s^{2} x^{2} \mathbf{1}_{\{ y \in \Delta_{Z_{s-}^{(\infty)}} \}} q_{\rm int}({\rm d} y, {\rm d} x) {\rm d} s & = \boldsymbol{\sigma}^{2}_{{\rm int}} \int_{0}^{t \wedge t^{\ast}} s^{2}Z_{s}^{(\infty)}(Z_{s}^{(\infty)}-1) {\rm d}s, \quad \text{for} \quad t \geq 0.
\end{align}
\noindent By Corollary \ref{corollaryII} \ref{PropC3} and \eqref{Assump4}, for any $0 \leq T <\infty$, there exist constants $0 < C_{1}, C_{2} < \infty$ such that 
\begin{align} \label{eq20}
\mathbb{E} \Big[ (\mathbf{m}_{\rm int})^{2} \int_{0}^{t \wedge t^{\ast}} \int_{\mathbb{N}} \int_{\mathbb{N}_{-1}}  s^{2} x^{2} \mathbf{1}_{\{ 0 < y \leq Z_{s-}^{(\infty)}\}} q_{\rm br}({\rm d} y, {\rm d} x) {\rm d} s  \Big] \leq C_{1} t^{2}
\end{align}
\noindent and
\begin{align} \label{eq20b}
\mathbb{E} \Big[ (\mathbf{m}_{\rm int})^{2} \int_{0}^{t \wedge t^{\ast}} \int_{\Delta} \int_{\mathbb{N}_{-1}}  s^{2} x^{2} \mathbf{1}_{\{ y \in \Delta_{Z_{s-}^{(\infty)}} \}} q_{\rm int}({\rm d} y, {\rm d} x) {\rm d} s \Big] \leq C_{2}t, \quad \text{for} \quad t \leq T.
\end{align}
\noindent It follows from standard properties of Poisson integration \cite[Theorem 8.23]{Peszat2007} that $(M_{t}^{\rm br})_{t \geq 0}$ and $(M_{t}^{\rm int})_{t \geq 0}$ are well-defined locally square-integrable martingales with quadratic variation \eqref{eq18} and \eqref{eq18b}, respectively. For $p = 2$, \eqref{eq19} and \eqref{eq19b} are consequence of \eqref{eq20}, \eqref{eq20b} and Doob's maximal inequality. For $p \in (0,2)$, \eqref{eq19} and \eqref{eq19b} follow by Jensen's inequality. This completes the proof.
\end{proof}

\begin{lemma}  \label{lemma11}
Suppose that $\mathbf{m}_{\rm int}<0$ and \eqref{Assump4} hold. For $p \in (0,2]$ and any $0 \leq T < T(\infty)$, there exists a constant $0 < C(p) <\infty$ such that 
\begin{align}
\mathbb{E} \Big[  \sup_{s \leq t} \Big| -\mathbf{m}_{\rm int}  s Z_{s}^{(\infty)}-1 \Big|^{p} \Big] \leq  C(p) (t^{p/2} \vee t^{p}), \quad \text{for} \quad t \leq T. 
\end{align}
\end{lemma}

\begin{proof}
We first note that, for $0 < r \leq t < T(\infty)$,
\begin{align}
- \int_{r}^{t} (\mathbf{m}_{\rm int} \mathbf{m}_{\rm br}sZ_{s}^{(\infty)} + \mathbf{m}_{\rm int} Z_{s}^{(\infty)} + (\mathbf{m}_{\rm int})^{2}s Z_{s}^{(\infty)} (Z_{s}^{(\infty)}-1) ) {\rm d} s 
\end{align}
\noindent can be written as
\begin{align}
- \int_{r}^{t} s^{-1}(- \mathbf{m}_{\rm int}sZ_{s}^{(\infty)})(- \mathbf{m}_{\rm int}sZ_{s}^{(\infty)}-1){\rm d} s  +\int_{r}^{t} (\mathbf{m}_{\rm int})^{2} s Z_{s}^{(\infty)} {\rm d} s  - \int_{r}^{t}  \mathbf{m}_{\rm int} \mathbf{m}_{\rm br} sZ_{s}^{(\infty)}{\rm d}s.
\end{align}
\noindent Fix $0 \leq T < T(\infty)$. Let $M_{t} = M_{t}^{\rm br} + M_{t}^{\rm int}$, for $t \geq 0$, where $(M_{t}^{\rm br})_{t \geq 0}$ and $(M_{t}^{\rm int})_{t \geq 0}$ are the martingales defined in \eqref{eq31} and \eqref{eq31b}, respectively, such that $T \leq t^{\ast} < T(\infty)$. By \eqref{eq17} and Lemma 10 in \cite{Bere2010} (with $g(s) = - \mathbf{m}_{\rm int}Z_{s}^{(\infty)}(- \mathbf{m}_{\rm int}sZ_{s}^{(\infty)}-1)$), we have that
\begin{align}
& \sup_{r \leq s \leq t} \left|- \mathbf{m}_{\rm int}sZ_{s}^{(\infty)}-1\right| \nonumber \\
& \quad \quad \quad \leq 2 \Big( \left|- \mathbf{m}_{\rm int}rZ_{r}^{(\infty)}-1\right| + |M_{r}| + \sup_{r \leq s \leq t} |M_{s}|  +\int_{0}^{t} | (\mathbf{m}_{\rm int})^{2} - \mathbf{m}_{\rm int} \mathbf{m}_{\rm br}|  s Z_{s}^{(\infty)} {\rm d} s \Big), 
\end{align}
\noindent for $r \leq t \leq T$. Then, Corollary \ref{corollaryII} \ref{PropC3}, \eqref{eq19} and \eqref{eq19b} imply that there are constants $0 < C_{1}(p), C_{2}(p) < \infty$ such that
\begin{align}
& \mathbb{E}\Big[\sup_{r \leq s \leq t} \left|- \mathbf{m}_{\rm int}sZ_{s}^{(\infty)}-1\right|^{p} \Big] \nonumber \\
& \quad \quad \quad \leq 2 \cdot 4^{p} \Big ( \mathbb{E}\Big[ \left|- \mathbf{m}_{\rm int}rZ_{r}^{(\infty)}-1\right|^{p} \Big] + \mathbb{E}[|M_{r}|^{p}] + C_{1}(p) t^{p/2} + C_{2}(p) t^{p} \Big). 
\end{align}
\noindent Our claim follows by letting $r \downarrow 0$, and using Theorem \ref{PropC2} and  the inequalities \eqref{eq19} and \eqref{eq19b}.
\end{proof}

\begin{lemma} \label{lemma12}
Suppose that $\mathbf{m}_{\rm int}<0$ and \eqref{Assump4} hold. Fix $0 < t^{\ast} < T(\infty)$. For all $t \geq 0$, the integral
\begin{align}
A_{t} = - \int_{0}^{t \wedge t^{\ast}} \Big(\mathbf{m}_{\rm int} \mathbf{m}_{\rm br}sZ_{s}^{(\infty)} + \mathbf{m}_{\rm int} Z_{s}^{(\infty)} + (\mathbf{m}_{\rm int})^{2}s Z_{s}^{(\infty)} (Z_{s}^{(\infty)}-1)\Big) {\rm d} s
\end{align}
\noindent is a well-defined Lebesgue integral, almost surely. Moreover,
\begin{align}
A_{t} = - \int_{0}^{t\wedge t^{\ast}} s^{-1}\Big(-s\mathbf{m}_{\rm int}Z_{s}^{(\infty)} - 1 \Big) {\rm d} s  + U_{t}, \quad \text{for} \quad t \geq 0,
\end{align}
\noindent where $(U_{t})_{t \geq 0}$ is a continuous stochastic process that satisfies the following: for any $0 \leq T <\infty$, there exists a constant $0 < C <\infty$ such that 
\begin{align} \label{eq26}
\mathbb{E}\Big[  \sup_{s \leq t} | U_{s} | \Big] \leq  C (t \vee t^{2}), \quad \text{for} \quad t \leq T.
\end{align}
\end{lemma}

\begin{proof}
For $t \geq 0$, 
\begin{align} \label{eq21}
A_{t} & = - \int_{0}^{t \wedge t^{\ast}} s^{-1}(- \mathbf{m}_{\rm int}sZ_{s}^{(\infty)}-1)^{2}{\rm d} s  - \int_{0}^{t \wedge t^{\ast}} s^{-1}(- \mathbf{m}_{\rm int}sZ_{s}^{(\infty)}-1) {\rm d} s \nonumber \\
& \quad \quad \quad +\int_{0}^{t \wedge t^{\ast}} (\mathbf{m}_{\rm int})^{2} s Z_{s}^{(\infty)} {\rm d} s  - \int_{0}^{t \wedge t^{\ast}}  \mathbf{m}_{\rm int} \mathbf{m}_{\rm br} sZ_{s}^{(\infty)}{\rm d}s.
\end{align}
\noindent Corollary \ref{corollaryII} \ref{PropC3} and Lemma \ref{lemma11} readily imply that each of the integrals in \eqref{eq21} is a well-defined Lebesgue integral for all $t\geq 0$ simultaneously, almost surely. This proves the first claim. 

By letting,
\begin{align}
U_{t} & = - \int_{0}^{t\wedge t^{\ast}} s^{-1}(- \mathbf{m}_{\rm int}sZ_{s}^{(\infty)}-1)^{2}{\rm d} s +\int_{0}^{t \wedge t^{\ast}} (\mathbf{m}_{\rm int})^{2} s Z_{s}^{(\infty)} {\rm d} s   - \int_{0}^{t \wedge t^{\ast}}  \mathbf{m}_{\rm int} \mathbf{m}_{\rm br} sZ_{s}^{(\infty)}{\rm d}s,
\end{align}
\noindent for $t \geq 0$, we deduce \eqref{eq26} by using Corollary \ref{corollaryII} \ref{PropC3} and Lemma \ref{lemma11}. This completes the proof.
\end{proof}

\begin{lemma}  \label{lemma13}
Suppose that  $\mathbf{m}_{\rm int}<0$ and \eqref{Assump4} hold. Let $(M_{t}^{\rm br})_{t \geq 0}$, $(M_{t}^{\rm int})_{t \geq 0}$  and $(U_{t})_{t \geq 0}$ be the processes defined in Lemma \ref{lemma10} and Lemma \ref{lemma12}, respectively, for some $0 < t^{\ast} < T(\infty)$. Then,
\begin{align}  \label{eq22}
- \mathbf{m}_{\rm int} t Z_{t}^{(\infty)} & = 1  - \int_{0}^{t} s^{-1}\Big(-s\mathbf{m}_{\rm int}Z_{s}^{(\infty)} - 1 \Big) {\rm d} s + M_{t}^{\rm br} + M_{t}^{\rm int} + U_{t}, \quad 0 \leq t \leq t^{\ast}.
\end{align}
\end{lemma}

\begin{proof}
By Lemma \ref{lemma10} and Lemma \ref{lemma12}, the process introduced in \eqref{eq22} is well-defined. On the other hand, by \eqref{eq17}, Lemma \ref{lemma10} and Lemma \ref{lemma12}, we have that, for $0 < r \leq t \leq t^{\ast}$, 
\begin{align}
- \mathbf{m}_{\rm int} t Z_{t}^{(\infty)} = - \mathbf{m}_{\rm int} r Z_{r}^{(\infty)} + (A_{t}-A_{r}) + (M_{t}^{\rm br}- M_{r}^{\rm br}) + (M_{t}^{\rm int}- M_{r}^{\rm int}). 
\end{align}
\noindent First, Theorem \ref{PropC2} implies that $- \mathbf{m}_{\rm int} r Z_{r}^{(\infty)}$ converges to $1$ in $L_{2}$, as $r \downarrow 0$. Second, by \eqref{eq19} and \eqref{eq19b}, we have that $M_{r}^{\rm br}$ and $M_{r}^{\rm int}$ converge to $0$ in $L_{2}$, as $r \downarrow 0$. Finally, \eqref{eq21}, Corollary \ref{corollaryII} \ref{PropC3} and Lemma \ref{lemma11} imply that there is a constant $0 < C <\infty$ such that 
\begin{align}
\mathbb{E}[ |A_{r}|] \leq C (r^{1/2} \vee r \vee r^{2}),
\end{align}
which implies that $A_{r}$ converges to $0$ in $L_{1}$, as $r \downarrow 0$. The above implies that for any fixed $0 < t \leq t^{\ast}$, as $r \downarrow 0$, both the left-hand side and the right-hand side of \eqref{eq17} converge in probability to the corresponding left and right hand side of \eqref{eq22}. Due to the c\`adl\`ag property of all the processes under consideration, identity \eqref{eq22} holds for all $0 \leq t \leq t^{\ast}$ simultaneously.
\end{proof}

Henceforth, we fix $0 < t^{\ast}< \infty$. Let $(D_{t})_{t \geq 0}$ be the stochastic process given by
\begin{align}
D_{t} := - \int_{0}^{t} s {\rm d} M_{s}^{\rm int}, \quad \text{for} \quad t \geq 0.  
\end{align}
\noindent Clearly, $(D_{t})_{t \geq 0}$ is well-defined. In particular, by \eqref{eq31b},
\begin{align}
D_{t} =  \mathbf{m}_{\rm int}  \int_{0}^{t \wedge t^{\ast}} \int_{\Delta} \int_{\mathbb{N}_{-1}}  s^{2} x \mathbf{1}_{\{ y \in \Delta_{Z_{s-}^{(\infty)}} \}} \tilde{N}_{\rm int}({\rm d} s, {\rm d} y, {\rm d} x), \quad \text{for} \quad t \geq 0. 
\end{align} 
\noindent For $k \in \mathbb{R}_{+}$, define $(M_{k, t}^{\rm int})_{t \geq 0}$ by letting
\begin{align}
M_{k, t}^{\rm int}:= - \mathbf{m}_{\rm int}  \int_{0}^{t \wedge t^{\ast}} \int_{\Delta} \int_{\mathbb{N}_{-1}}  s x \mathbf{1}_{\{|x| \leq k \}} \mathbf{1}_{\{ y \in \Delta_{Z_{s-}^{(\infty)}} \}} \tilde{N}_{\rm int}({\rm d} s, {\rm d} y, {\rm d} x), \quad \text{for} \quad t \geq 0.
\end{align}
\noindent Note that $(M_{k, t}^{\rm int})_{t \geq 0}$ is a well-defined locally square-integrable martingale (see e.g., the argument used in the proof of Lemma \ref{lemma10}). Then, define the stochastic process $(D_{k, t})_{t \geq 0}$ as follows
\begin{align} \label{eq32}
D_{k, t} := - \int_{0}^{t } s {\rm d} M_{k, t}^{\rm int} = \mathbf{m}_{\rm int}  \int_{0}^{t \wedge t^{\ast}} \int_{\Delta} \int_{\mathbb{N}_{-1}}  s^{2} x \mathbf{1}_{\{|x| \leq k \}} \mathbf{1}_{\{ y \in \Delta_{Z_{s-}^{(\infty)}} \}} \tilde{N}_{\rm int}({\rm d} s, {\rm d} y, {\rm d} x), \quad \text{for} \quad t \geq 0.  
\end{align}

\begin{lemma} \label{lemma15}
Suppose that $\mathbf{m}_{\rm int}<0$ and \eqref{Assump4} hold. For $k \geq 1$ and $0 \leq T < \infty$, there exists a constant $0 < C < \infty$ such that 
\begin{align}
\mathbb{E} \Big[ \sup_{s \leq t} (D_{s}- D_{k, s})^{2} \Big] \leq C \boldsymbol{\sigma}^{2}_{{\rm int}, k} t^{3}, \quad \text{for} \quad t \leq T,
\end{align}
\noindent where $\boldsymbol{\sigma}^{2}_{{\rm int},k} \coloneqq \sum_{i > \lfloor k \rfloor} i^{2}b_{i}$.
\end{lemma}

\begin{proof}
Observe that $(D_{k, t}-D_{t})_{t \geq 0}$ is a locally square-integrable martingale with predictable quadratic variation $(\langle D_{k, \cdot}-D_{\cdot} \rangle_{t})_{t \geq 0}$ given by
\begin{align} 
\langle D_{k, \cdot}-D_{\cdot} \rangle_{t} = (\mathbf{m}_{\rm int})^{2} \boldsymbol{\sigma}^{2}_{{\rm int}, k}   \int_{0}^{t \wedge t^{\ast}} s^{4} Z_{s}^{(\infty)} \left(Z_{s}^{(\infty)}-1\right) {\rm d} s, \quad \text{for} \quad t \geq 0.
\end{align}
\noindent Then, for any $0 \leq T < \infty$, Corollary \ref{corollaryII} \ref{PropC3} implies that there exists a constant $0 < C < \infty$ such that 
\begin{align}
\mathbb{E}[ \langle D_{k, \cdot}-D_{\cdot} \rangle_{t} ] & \leq C  \boldsymbol{\sigma}^{2}_{{\rm int}, k} t^{3}, \quad \text{for} \quad t \leq T.
\end{align}
\noindent Finally, our claim follows by Doob's maximal inequality. 
\end{proof}

\begin{proposition} \label{prop3}
Suppose that $ \mathbf{m}_{\rm int}<0$ and \eqref{Assump4} hold. Let $\varepsilon >0$, the process $(D_{t}^{(\varepsilon)})_{t \geq 0}$ defined by 
\begin{align}
D_{t}^{(\varepsilon)} := -\varepsilon^{-\frac{3}{2}} D_{\varepsilon t},\qquad \textrm{ for }\quad t \geq 0,
\end{align}
converges weakly, in $\mathbb{D}(\mathbb{R}_{+}, \mathbb{R})$, as $\varepsilon \downarrow 0$, to a Gaussian process $(\widehat{X}_{t})_{t \geq 0}$ given by
\begin{align}
\widehat{X}_{0}=0 \qquad \text{and} \qquad  \widehat{X}_{t} = \boldsymbol{\sigma}_{{\rm int}} \int_{0}^{t} u {\rm d}W_{u}, \qquad \text{for} \quad t >0, 
\end{align}
\noindent where $(W_{t})_{t \geq 0}$ is a standard Brownian motion. 
\end{proposition}

\begin{proof}
For $k \geq 1$, let $(D_{k, t}^{(\varepsilon)})_{t \geq 0}$ be the stochastic process defined by $D_{k, t}^{(\varepsilon)} := -\varepsilon^{-\frac{3}{2}} D_{k, \varepsilon t}$, for $t \geq 0$. We claim that $(D_{k, t}^{(\varepsilon)})_{t \geq 0}$ converges weakly, in $\mathbb{D}(\mathbb{R}_{+}, \mathbb{R})$, as $\varepsilon \downarrow 0$, to a Gaussian process $(\widehat{X}_{t}^{(k)})_{t \geq 0}$ given by
\begin{align} \label{eq33}
\widehat{X}_{0}^{(k)}=0 \qquad \text{and} \qquad  \widehat{X}_{t}^{(k)} = \hat{\boldsymbol{\sigma}}_{{\rm int}, k}\int_{0}^{t} u {\rm d}W_{u}, \qquad \text{for} \quad t >0, 
\end{align}
\noindent where $\hat{\boldsymbol{\sigma}}^{2}_{{\rm int}, k} \coloneqq c+ \sum_{i =1}^{\lfloor k \rfloor} i^{2}b_{i}$. Recall that convergence in $\mathbb{D}(\mathbb{R}_{+}, \mathbb{R})$ to a continuous limit is equivalent to uniform convergence on compact subsets of $[0, \infty)$, see for instance Proposition 1.17 in Chapter VI of \cite{JaShi03}. Therefore, since by \eqref{Assump4}, $\hat{\boldsymbol{\sigma}}^{2}_{{\rm int}, k} \rightarrow \boldsymbol{\sigma}^{2}_{{\rm int}}$, as $k \rightarrow \infty$, it is not difficult to see that Lemma \ref{lemma15} and \eqref{eq33} imply the result in Proposition \ref{prop3}.

To prove that $(D_{k, t}^{(\varepsilon)})_{t \geq 0}$  converges weakly towards $(\widehat{X}_{t}^{(k)})_{t \geq 0}$, we use Theorem 1.4 in Chapter 7.1 of \cite{Et1986}. Note that $(D_{k, t}^{(\varepsilon)})_{t \geq 0}$ is a well-defined locally square-integrable martingale. In particular, by \eqref{eq32}, its predictable quadratic variation $(\langle D_{k, \cdot}^{(\varepsilon)} \rangle_{t})_{t \geq 0}$ is given by
\begin{align}  \label{eq34}
\langle D_{k, \cdot}^{(\varepsilon)} \rangle_{t}   & =   \varepsilon^{-3}  (\mathbf{m}_{\rm int})^{2} \hat{\boldsymbol{\sigma}}^{2}_{{\rm int}, K} \int_{0}^{\varepsilon t \wedge t^{\ast}} s^{4}Z_{s}^{(\infty)}(Z_{s}^{(\infty)}-1) {\rm d} s \nonumber \\
& =  \hat{\boldsymbol{\sigma}}^{2}_{{\rm int}, K} \int_{0}^{t \wedge (t^{\ast}/\varepsilon)} s^{2}(- \mathbf{m}_{\rm int} \varepsilon s)^{2}Z_{ \varepsilon s}^{(\infty)}(Z_{\varepsilon s}^{(\infty)}-1) {\rm d} s.
\end{align}
\noindent We verify assumptions (b) Theorem 1.4 in Chapter 7.1 of \cite{Et1986} with 
\begin{align}
c_{1,1}(t) = \hat{\boldsymbol{\sigma}}^{2}_{{\rm int}, k} \int_{0}^{t} s^{2}{\rm d} s\qquad \textrm{and}\qquad A_{n}^{11}(t) = \langle D_{k, \cdot}^{(\varepsilon)} \rangle_{t}, \qquad\textrm{for} \qquad t \geq 0.
\end{align}
Since $(\langle D_{k, \cdot}^{(\varepsilon)} \rangle_{t})_{t \geq 0}$ is a continuous process, (1.16) in Theorem 1.4 in Chapter 7.1 of \cite{Et1986} is satisfied. Moreover, Theorem \ref{PropC2}, Corollay \ref{corollaryII} \ref{PropC3} and \eqref{eq34} imply that, for each fixed $t \geq 0$, 
\begin{align}
\langle D_{k, \cdot}^{(\varepsilon)} \rangle_{t}\quad \textrm{converges to  }\quad\hat{\boldsymbol{\sigma}}^{2}_{{\rm int}, k} \int_{0}^{t} s^{2}{\rm d} s,\qquad \textrm{as}\qquad \varepsilon \downarrow 0,
\end{align}
 in probability (this is, (1.19) in Theorem 1.4 in Chapter 7.1 of \cite{Et1986}). Finally, from the definition of $(D_{k, t}^{(\varepsilon)})_{t \geq 0}$ and by using \eqref{eq32}, it follows that  its jumps on $[0,T]$, for $0 \leq T < \infty$, are uniformly bounded by $- \mathbf{m}_{\rm int} \varepsilon^{1/2}T^{2}k$, which implies (1.17) in Theorem 1.4 in Chapter 7.1 of \cite{Et1986}. This concludes our proof. 
\end{proof}

Let $(Y_{t})_{t \geq 0}$ be the process given by
\begin{align}
Y_{0}=0 \qquad \text{and} \qquad  Y_{t} := t^{-1} \int_{0}^{t} s {\rm d} M_{s}^{\rm int}, \quad \text{for} \quad t >0.  
\end{align}
\noindent Clearly, $(Y_{t})_{t \geq 0}$ is well-defined. For $\varepsilon >0$, define the stochastic process $(Y_{t}^{(\varepsilon)})_{t \geq 0}$ by letting $Y_{t}^{(\varepsilon)} = \varepsilon^{-\frac{1}{2}} Y_{\varepsilon t}$, for $t \geq 0$. Recall the definition of the stochastic process $(X_{t}^{(\varepsilon)})_{t \geq 0}$ in \eqref{FlucD}. 

\begin{lemma} \label{lemma14}
Suppose that  $ \mathbf{m}_{\rm int}<0$ and \eqref{Assump4} hold. The process $(Y_{t})_{t \geq 0}$ satisfies 
\begin{align} \label{eq23}
 Y_{t} = - \int_{0}^{t} s^{-1}Y_{s}{\rm d} s+M_{t}^{\rm int}, \quad \text{for} \quad t \geq 0.  
\end{align}
Moreover, for any $0 \leq T <\infty$, there exists a constant $0 < C <\infty$ such that 
\begin{align} \label{eq24}
\mathbb{E} \Big[  \sup_{s \leq t} Y_{s}^{2} \Big] \leq  C (t \vee t^{2}), \quad \text{for} \quad t \leq T,
\end{align}
\noindent and
\begin{align} \label{eq25}
\lim_{\varepsilon \downarrow 0}\mathbb{E} \Big[  \sup_{t \leq T} \Big |X_{t}^{(\varepsilon)} -Y_{t}^{(\varepsilon)} \Big| \Big] = 0.
\end{align}
\end{lemma}

\begin{proof}
By Corollary \ref{corollaryII} \ref{PropC3}, the predictable quadratic variation $(\langle D \rangle_{t})_{t \geq 0}$ of $(D_{t})_{t \geq 0}$ satisfies
\begin{align}
\mathbb{E}[\langle D \rangle_{t}] & = (\mathbf{m}_{\rm int})^{2} \boldsymbol{\sigma}^{2}_{{\rm int}} \mathbb{E}\Big[ \int_{0}^{t \wedge t^{\ast}} s^{4}Z_{s}^{(\infty)}(Z_{s}^{(\infty)}-1) {\rm d} s \Big] \leq C_{1} \int_{0}^{t} s^{2} {\rm d}s = \frac{C_{1}}{3}t^{3},
\end{align}
\noindent for some constant $0 < C_{1} < \infty$. In particular,
\begin{align} \label{eq28}
\mathbb{E}[Y_{t}^{2}] = t^{-2} \mathbb{E}[D_{t}^{2}] = t^{-2} \mathbb{E}[\langle D \rangle_{t}] \leq \frac{C_{1}}{3}t.
\end{align}

The identity \eqref{eq23} follows by integration by parts (note that $t \mapsto 1/t$ is continuous and of finite variation on any closed interval $[a, b]$, for $0 < a < b$). The only subtle point is the lack of regularity of $t \mapsto 1/t$ at $0$. This technicality is overcome, by writing first the formula for $Y_{t}-Y_{r}$, for any $0< r \leq t$, i.e.,
\begin{align} \label{eq29}
 Y_{t} -Y_{r}= - \int_{r}^{t} s^{-2} \int_{0}^{s} u {\rm d} M_{u}^{\rm int} {\rm d} s + \int_{(r,t]} {\rm d} M_{s}^{\rm int} = - \int_{r}^{t} s^{-1}Y_{s} {\rm d} s + M_{t}^{\rm int} - M_{r}^{\rm int},
\end{align}
\noindent and then letting $r \downarrow 0$. Indeed, by using \eqref{eq28}, we see that there exists a constant $0 < C_{2} < \infty$ such that $\mathbb{E}[|\int_{0}^{t} s^{-1}Y_{s}{\rm d s}| ] \leq C_{2}t^{1/2}$ and thus, $\int_{0}^{t} s^{-1}Y_{s}{\rm d s}$ is a well-defined Lebesgue integral for all $t \geq 0$ simultaneously, almost surely.

The estimate \eqref{eq24} follows from \eqref{eq23}, Lemma 10 in \cite{Bere2010} and \eqref{eq19b} in Lemma \ref{lemma10}. Observe that by \eqref{FlucD}, \eqref{eq22} and \eqref{eq23} we have that
\begin{align}
X_{t}^{(\varepsilon)} - Y_{t}^{(\varepsilon)} = - \varepsilon^{-1/2} \int_{0}^{t} s^{-1} (X_{s}^{(\varepsilon)}-Y_{s}^{(\varepsilon)}) {\rm d} s + \varepsilon^{-1/2}  M_{\varepsilon t}^{\rm br} + \varepsilon^{-1/2}  U_{\varepsilon t}, \quad \text{for} \quad 0 \leq t \leq t^{\ast}/\varepsilon.
\end{align}
\noindent By Lemma 10 in \cite{Bere2010}, $\sup_{s \leq t} |X_{s}^{(\varepsilon)} - Y_{s}^{(\varepsilon)}| \leq 2 \varepsilon^{-1/2} \sup_{s \leq t} |M_{\varepsilon s}^{\rm br} + U_{\varepsilon s}|$.  Then, by \eqref{eq19} and \eqref{eq26}, there exists a constant $0 < C < \infty$ such that $\mathbb{E}[ \sup_{s \leq T} |X_{s}^{(\varepsilon)} - Y_{s}^{(\varepsilon)}| ] \leq C ( \varepsilon^{1/2} T \vee \varepsilon^{3/2} T^{2})$, for $0 \leq T <\infty$ and $\varepsilon >0$ such that $T \leq t^{\ast}/\varepsilon$, which implies \eqref{eq25}. This completes the proof.
\end{proof}

Finally, we prove Theorem \ref{Main3}. 

\begin{proof}[Proof of Theorem \ref{Main3}]
By \eqref{eq25} and the symmetry of the law of the standard Brownian motion $(W_{t})_{t \geq 0}$ it suffices to show that the process $(-Y^{(\varepsilon)}_{t})_{t \geq 0}$, converges weakly, in $\mathbb{D}(\mathbb{R}_{+}, \mathbb{R})$, as $\varepsilon \downarrow 0$, to the process $(X_{t})_{t \geq 0}$. To do that, we follow the argument explained in the last part of the proof of Theorem 1.1 of \cite{Vlada20152}, which in turn is based in Steps 2-4 of the proof of Lemma 4.8 of \cite{Vlada2015}. For $k >0$, define the stochastic process $(Y^{(\varepsilon)}_{k, t})_{t \geq 0}$ by letting 
\begin{align}
Y^{(\varepsilon)}_{k, t} := -(k^{-1} \mathbf{1}_{\{t \in [0,k]\}} + t^{-1}  \mathbf{1}_{\{t \in (k,\infty)\}})D_{t}^{(\varepsilon)}, \quad \text{for} \quad t \geq 0.
\end{align}
\noindent Recall that if $f: \mathbb{R}_{+} \rightarrow \mathbb{R}$ is continuous, then the mapping $w \mapsto f(w)$ is continuous from $\mathbb{D}(\mathbb{R}_{+}, \mathbb{R})$ into itself. Hence, Proposition \ref{prop3} implies that for any $k >0$, as $\varepsilon \downarrow 0$, the process $(Y^{(\varepsilon)}_{k, t})_{t \geq 0}$ converges weakly, in $\mathbb{D}(\mathbb{R}_{+}, \mathbb{R})$, to the process $(\tilde{X}_{t}^{(k)})_{t \geq 0}$ defined by
\begin{align}
\tilde{X}^{(k)}_{t} := -(tk^{-1} \mathbf{1}_{\{t \in [0,k]\}} +  \mathbf{1}_{\{t \in (k,\infty)\}}) X_{t}, \quad \text{for} \quad t \geq 0.
\end{align}

Note that $Y^{(\varepsilon)}_{t} = - t^{-1} D_{t}^{(\varepsilon)}$ for $t >0$, and thus, $Y^{(\varepsilon)}_{t}-Y^{(\varepsilon)}_{k, t} = (k^{-1}-t^{-1}) \mathbf{1}_{\{t \in (0,k]\}} D_{t}^{(\varepsilon)}$, for $t \geq 0$. By \eqref{eq24}, there exists a constant $0 < C < \infty$ (independent of $\varepsilon$) such that
\begin{align} \label{eq35}
\mathbb{E} \Big[ \sup_{t \geq 0}|Y^{(\varepsilon)}_{t}-Y^{(\varepsilon)}_{k, t}|^{2} \Big] \leq 4  \mathbb{E}\Big[ \sup_{t \leq k}|Y^{(\varepsilon)}_{t}|^{2} \Big]  \leq  C (k \vee k^{2}).
\end{align}
\noindent On the other hand, $-X_{t} - \tilde{X}_{t}^{(k)} = (tk^{-1}-1) \mathbf{1}_{\{t \in [0,k]\}} X_{t}$, for $t \geq 0$. By Doob's maximal inequality, there exists a constant $0 < C < \infty$ (independent of $\varepsilon$) such that
\begin{align} \label{eq36}
\mathbb{E} \Big[ \sup_{t \geq 0}|-X_{t} - \tilde{X}_{t}^{(k)}|^{2} \Big] \leq  4 \mathbb{E} \Big[ \sup_{t \leq k}|X_{t}|^{2} \Big]  \leq  C k.
\end{align}

Let $d_{\infty}^{\rm Sk}$ denote the Skorokhod metric on $\mathbb{D}(\mathbb{R}_{+}, \mathbb{R})$ (here, $\mathbb{R}$ is equipped with the Euclidean distance) as defined in Section 1 of Chapter VI in \cite{JaShi03}. In particular, note that $d_{\infty}^{\rm Sk}(f,g) \leq \sup_{t \geq 0}|f(t) -g(t)|$, for any $f,g \in \mathbb{D}(\mathbb{R}_{+}, \mathbb{R})$. Finally, it is straightforward to show that $(-Y^{(\varepsilon)}_{t})_{t \geq 0}$ converges to $(X_{t})_{t \geq 0}$. By Theorem 2.1 in Chapter 1 of \cite{Billingsley1999}, it suffices to show that
\begin{align}
\lim_{\varepsilon \downarrow 0} | \mathbb{E}[F(-Y^{(\varepsilon)}_{t})_{t \geq 0})] - \mathbb{E}[F((X_{t})_{t \geq 0})] | = 0,
\end{align}
\noindent for all all bounded, uniformly continuous function $F: \mathbb{D}(\mathbb{R}_{+}, \mathbb{R}) \rightarrow \mathbb{R}$. The above follows from the convergence of $(Y^{(\varepsilon)}_{k, t})_{t \geq 0}$ towards $(\tilde{X}_{t}^{(k)})_{t \geq 0}$ (i.e., for any $k >0$, $|\mathbb{E}[F((Y^{(\varepsilon)}_{k, t})_{t \geq 0})] - \mathbb{E}[F((\tilde{X}_{t}^{(k)})_{t \geq 0})] | \rightarrow 0$, as $\varepsilon \downarrow 0$), \eqref{eq35}, \eqref{eq36}, the triangle inequality, the uniform continuity of $F$, the Markov inequality and the above discussion (the argument is based on the addition and subtraction of intermediate terms). This concludes the proof.
\end{proof}

\appendix
\section{Appendix} \label{Apendice}

The next lemma is a modification of Lemma 2.2 in \cite{Van1995}. Let $(\Omega, \mathcal{F}, \mathcal{P})$ be a probability space and let $(M_{t})_{t \geq 0}$ be a locally square-integrable martingale w.r.t.\ the filtration $(\mathcal{F}_{t})_{t \geq 0}$. Assume that $M_{0}=0$ and that $(\mathcal{F}_{t})_{t \geq 0}$ satisfies the usual conditions. Let $(V^{(2)}_{t})_{t \geq 0}$ be the predictable variation of $(M_{t})_{t \geq 0}$, and for $p \geq 3$ an integer, let $(V^{(p)}_{t})_{t \geq 0}$ be the (predictable) compensator of $(\sum_{s \leq t} |\Delta M_{s}|^{p})_{t\geq 0}$. 

\begin{lemma} \label{lemmaA1}
Suppose that there exists a (deterministic) function $f: \mathbb{R}_{+} \rightarrow (0, \infty)$ such that for all $t \geq 0$ and some $0< K<\infty$, 
\begin{align} \label{eqA1}
V^{(p)}_{t} \leq \frac{p!}{2} K^{p-2} f(t), \quad \text{for} \quad p=2,3,\dots.
\end{align}
\noindent For $0 < \lambda < 1/K$, let $L_{t} \coloneqq \sum_{p=2}^{\infty} \frac{\lambda^{p}}{p!} V^{(p)}_{t}$, for $t \geq 0$. Then, 
\begin{align}
(\exp(\lambda M_{t}-L_{t}))_{t \geq 0}, \quad \text{is a super-martingale}.
\end{align}
\noindent Furthermore, for each $x >0$ and $0 \leq T < \infty$,
\begin{align}
\mathbb{P} \Big(\sup_{t \leq T} M_{t} \geq x \Big) \leq e^{-\frac{x^{2}}{2(xK+f(T))}}.
\end{align}
\end{lemma}

\begin{proof}
By \eqref{eqA1}, 
\begin{align} \label{eqA2}
L_{t}  \leq \frac{\lambda^{2} f(t)}{2(1-\lambda K)}, \quad \text{for} \quad t \geq 0. 
\end{align} 
\noindent It is has been proven in the Proof of Lemma 2.2 in \cite{Van1995} that $(\exp(\lambda M_{t}-L_{t}))_{t \geq 0}$ is a super-martingale. Observe that $(L_{t})_{t \geq 0}$ is an increasing process. Then, it follows from \eqref{eqA2} that
\begin{align}
\mathbb{P} \Big(\sup_{t \leq T} M_{t} \geq x \Big) = \mathbb{P} \Big(\sup_{t \leq T} e^{\lambda M_{t} - L_{t}} \geq e^{\lambda x-\frac{\lambda^{2} f(T)}{2(1-\lambda K)}} \Big) \leq e^{-\lambda x+\frac{\lambda^{2} f(T)}{2(1-\lambda K)}},
\end{align}
\noindent for $0 < \lambda < 1/K$. Finally, our second claim follows by choosing $\lambda = x/(f(T)+ Kx)$.
\end{proof}

\paragraph{Acknowledgements.}
JCP gratefully acknowledges the financial support provided by CONAHCyT through the Ciencia de Frontera grant CF-2023-I-2566.


\begin{thebibliography}{10}

\bibitem{Alkemper2007}
R.~Alkemper and M.~Hutzenthaler, \emph{Graphical representation of some duality
  relations in stochastic population models}, Electron. Comm. Probab.
  \textbf{12} (2007), 206--220. \MR{2320823}

\bibitem{Athreya2005}
S.~R. Athreya and J.~M. Swart, \emph{Branching-coalescing particle systems},
  Probab. Theory Related Fields \textbf{131} (2005), no.~3, 376--414.
  \MR{2123250}

\bibitem{Bansaye2019}
V.~Bansaye, \emph{Approximation of stochastic processes by nonexpansive flows
  and coming down from infinity}, Ann. Appl. Probab. \textbf{29} (2019), no.~4,
  2374--2438. \MR{3984256}

\bibitem{Bansaye2016}
V.~Bansaye, S.~M\'{e}l\'{e}ard, and M.~Richard, \emph{Speed of coming down from
  infinity for birth-and-death processes}, Adv. in Appl. Probab. \textbf{48}
  (2016), no.~4, 1183--1210. \MR{3595771}

\bibitem{BerestyckiJ2004}
J.~Berestycki, \emph{Exchangeable fragmentation-coalescence processes and their
  equilibrium measures}, Electron. J. Probab. \textbf{9} (2004), no. 25,
  770--824. \MR{2110018}

\bibitem{Bere2010}
J.~Berestycki, N.~Berestycki, and V.~Limic, \emph{The {$\Lambda$}-coalescent
  speed of coming down from infinity}, Ann. Probab. \textbf{38} (2010), no.~1,
  207--233. \MR{2599198}

\bibitem{Berzunza2020}
G.~Berzunza~Ojeda and J.~C. Pardo, \emph{Branching processes with pairwise
  interactions}, ALEA Lat. Am. J. Probab. Math. Stat. \textbf{22} (2025),
  no.~1, 711--748. \MR{4929072}

\bibitem{Billingsley1999}
P.~Billingsley, \emph{Convergence of probability measures}, second ed., Wiley
  Series in Probability and Statistics: Probability and Statistics, John Wiley
  \& Sons, Inc., New York, 1999, A Wiley-Interscience Publication. \MR{1700749}

\bibitem{BJS2015}
J.~E. Bj\"{o}rnberg and S.~O. Stef\'{a}nsson, \emph{Random walk on random
  infinite looptrees}, J. Stat. Phys. \textbf{158} (2015), no.~6, 1234--1261.
  \MR{3317412}

\bibitem{Et1986}
S.~N. Ethier and T.~G. Kurtz, \emph{Markov processes}, Wiley Series in
  Probability and Mathematical Statistics: Probability and Mathematical
  Statistics, John Wiley \& Sons, Inc., New York, 1986, Characterization and
  convergence. \MR{838085}

\bibitem{Foucart2022}
C.~Foucart, \emph{A phase transition in the coming down from infinity of simple
  exchangeable fragmentation-coagulation processes}, Ann. Appl. Probab.
  \textbf{32} (2022), no.~1, 632--664. \MR{4386538}

\bibitem{Zenghu2010}
Z.~Fu and Z.~Li, \emph{Stochastic equations of non-negative processes with
  jumps}, Stochastic Process. Appl. \textbf{120} (2010), no.~3, 306--330.
  \MR{2584896}

\bibitem{Pa2021}
A.~Gonz\'{a}lez~Casanova, J.~C. Pardo, and J.~L. P\'{e}rez, \emph{Branching
  processes with interactions: subcritical cooperative regime}, Adv. in Appl.
  Probab. \textbf{53} (2021), no.~1, 251--278. \MR{4232756}

\bibitem{GONZALEZ2020}
A.~{González Casanova}, V.~{Miró Pina}, and J.~C. Pardo, \emph{The
  {W}right-{F}isher model with efficiency}, Theoretical Population Biology
  \textbf{132} (2020), 33--46.

\bibitem{griffiths1997}
R.~C. Griffiths and P.~Marjoram, \emph{An ancestral recombination graph},
  Institute for Mathematics and its Applications \textbf{87} (1997), 257.

\bibitem{Hanson2020}
P.~A. {Hanson}, P.~A. {Jenkins}, J.~{Koskela}, and D.~{Span{\`o}},
  \emph{{Diffusion Limits at Small Times for Coalescent Processes with Mutation
  and Selection}}, arXiv e-prints (2020), arXiv:2012.10316.

\bibitem{Ikeda1981}
N.~Ikeda and S.~Watanabe, \emph{Stochastic differential equations and diffusion
  processes}, North-Holland Mathematical Library, vol.~24, North-Holland
  Publishing Co., Amsterdam-New York; Kodansha, Ltd., Tokyo, 1981. \MR{637061}

\bibitem{JaShi03}
J.~Jacod and A.~N. Shiryaev, \emph{Limit theorems for stochastic processes},
  second ed., Grundlehren der mathematischen Wissenschaften [Fundamental
  Principles of Mathematical Sciences], vol. 288, Springer-Verlag, Berlin,
  2003. \MR{1943877}

\bibitem{Jagers1994}
P.~Jagers, \emph{Towards dependence in general branching processes}, Classical
  and modern branching processes ({M}inneapolis, {MN}, 1994), IMA Vol. Math.
  Appl., vol.~84, Springer, New York, 1997, pp.~127--139. \MR{1601717}

\bibitem{Ka82}
A.~V. Kalinkin, \emph{On the probability of the extinction of branching process
  with interaction of particles}, Teor. Veroyatnost. i Primenen. \textbf{27}
  (1982), no.~1, 201--205.

\bibitem{Ka99}
A.~V. Kalinkin, \emph{Final probabilities of a branching process with interaction of
  particles, and the epidemic process}, Teor. Veroyatnost. i Primenen.
  \textbf{43} (1998), no.~4, 773--780. \MR{1692393}

\bibitem{Ka01}
A.~V. Kalinkin, \emph{Exact solutions of the {K}olmogorov equations for critical
  branching processes with two complexes of particle interaction}, Uspekhi Mat.
  Nauk \textbf{56} (2001), no.~3(339), 173--174. \MR{1859738}

\bibitem{Ka02}
A.~V. Kalinkin, \emph{On the probability of the extinction of a branching process with
  two complexes of particle interaction}, Teor. Veroyatnost. i Primenen.
  \textbf{46} (2001), no.~2, 376--381. \MR{1968693}

\bibitem{Ka02-1}
A.~V. Kalinkin, \emph{Markov branching processes with interaction}, Uspekhi Mat. Nauk
  \textbf{57} (2002), no.~2(344), 23--84. \MR{1918194}

\bibitem{Kingman1982}
J.~F.~C. Kingman, \emph{The coalescent}, Stochastic Process. Appl. \textbf{13}
  (1982), no.~3, 235--248. \MR{671034}

\bibitem{KRONE1997}
S.~M. Krone and C.~Neuhauser, \emph{Ancestral processes with selection},
  Theoretical Population Biology \textbf{51} (1997), no.~3, 210--237.

\bibitem{Lambert2005}
A.~Lambert, \emph{The branching process with logistic growth}, Ann. Appl.
  Probab. \textbf{15} (2005), no.~2, 1506--1535. \MR{2134113}

\bibitem{Vlada20152}
V.~Limic and A.~Talarczyk, \emph{Diffusion limits at small times for
  {$\Lambda$}-coalescents with a {K}ingman component}, Electron. J. Probab.
  \textbf{20} (2015), no. 45, 20. \MR{3339865}

\bibitem{Vlada2015}
V.~Limic and A.~Talarczyk, \emph{Second-order asymptotics for the block counting process in a
  class of regularly varying {$\Lambda$}-coalescents}, Ann. Probab. \textbf{43}
  (2015), no.~3, 1419--1455. \MR{3342667}

\bibitem{Neuhauser1997}
C.~Neuhauser and S.~M. Krone, \emph{The genealogy of samples in models with
  selection.}, Genetics \textbf{145 2} (1997), 519--34.

\bibitem{Palau2018}
S.~Palau and J.~C. Pardo, \emph{Branching processes in a {L}\'{e}vy random
  environment}, Acta Appl. Math. \textbf{153} (2018), 55--79. \MR{3745730}

\bibitem{Peszat2007}
S.~Peszat and J.~Zabczyk, \emph{Stochastic partial differential equations with
  {L}\'{e}vy noise}, Encyclopedia of Mathematics and its Applications, vol.
  113, Cambridge University Press, Cambridge, 2007, An evolution equation
  approach. \MR{2356959}

\bibitem{Schweinsberg2000}
J.~Schweinsberg, \emph{A necessary and sufficient condition for the
  {$\Lambda$}-coalescent to come down from infinity}, Electron. Comm. Probab.
  \textbf{5} (2000), 1--11. \MR{1736720}

\bibitem{Van1995}
S.~van~de Geer, \emph{Exponential inequalities for martingales, with
  application to maximum likelihood estimation for counting processes}, Ann.
  Statist. \textbf{23} (1995), no.~5, 1779--1801. \MR{1370307}

\end{thebibliography}

\providecommand{\bysame}{\leavevmode\hbox to3em{\hrulefill}\thinspace}
\providecommand{\MR}{\relax\ifhmode\unskip\space\fi MR }
\providecommand{\MRhref}[2]{%
  \href{http://www.ams.org/mathscinet-getitem?mr=#1}{#2}
}
\providecommand{\href}[2]{#2}

\end{document}